\documentclass[aop,MSNbibl,seceqn,dvips]{arximspdf}

\doi{10.1214/13-AOP906} 
\volume{42}
\issue{6}
\pubyear{2014}
\firstpage{2595}
\lastpage{2643}

\makeatletter
\newcommand{\rrvert}{\vert}
\newcommand{\llvert}{\vert}
\newcommand{\cA}{\mathcal{A}}
\newcommand{\cB}{\mathcal{B}}
\newcommand{\cG}{\mathcal{G}}
\newcommand{\cF}{\mathcal{F}}
\newcommand{\cS}{\mathcal{S}}
\newcommand{\cE}{\mathcal{E}}
\newcommand{\cK}{\mathcal{K}}
\newcommand{\cJ}{\mathcal{J}}
\newcommand{\cI}{\mathcal{I}}
\newcommand{\cD}{\mathcal{D}}
\newcommand{\cM}{\mathcal{M}}
\newcommand{\cN}{\mathcal{N}}
\newcommand{\Ex}{\mathbb{E}}
\newcommand{\id}{\mathbf{1}}
\newcommand{\vNT}{\otimes}
\newcommand{\oto}{}
\newcommand{\ot}{\otimes}
\newcommand{\tr}{\tau}
\newcommand{\al}{\alpha}
\newcommand{\Om}{\Omega}
\newcommand{\Si}{\Sigma}
\newcommand{\R}{\mathbb{R}}
\newcommand{\C}{\mathbb{C}}
\newcommand{\N}{\mathbb{N}}
\newcommand{\bP}{\mathbb{P}}
\newcommand{\E}{\mathbb{E}}
\newcommand{\eps}{\varepsilon}
\newcommand{\ti}{\times}

\newcommand{\epf}{}

\newcommand\tnorm[1]{\llvert\! \llvert\! \llvert
#1\rrvert\! \rrvert\! \rrvert }

\newproclaim{definition}{Definition}[section]
\newtheorem{theorem}[definition]{Theorem}

\newtheorem{corollary}[definition]{Corollary}
\newproclaim{remark}[definition]{Remark}
\newtheorem{lemma}[definition]{Lemma}
\newproclaim{assumption}[definition]{Assumption}
\newproclaim{example}{Example}[section]

\makeatother

\begin{document}
\begin{frontmatter}

\title{It\^{o} isomorphisms for $L^{\lowercase{p}}$-valued Poisson stochastic~integrals\thanksref{T1}}
\runtitle{It\^{o} isomorphisms $L^{\lowercase{p}}$-valued Poisson stochastic integrals}

\begin{aug}
\author[A]{\fnms{Sjoerd} \snm{Dirksen}\corref{}\ead[label=e1]{sjoerd.dirksen@hcm.uni-bonn.de}}
\affiliation{University of Bonn}

\address[A]{Hausdorff Center for Mathematics\\
University of Bonn\\
Endenicher Allee 60\\
53115 Bonn\\
Germany\\
\printead{e1}}
\runauthor{S. Dirksen}
\end{aug}
\thankstext{T1}{Supported by VICI subsidy 639.033.604 of the
Netherlands Organisation for Scientific Research (NWO) and the
Hausdorff Center for Mathematics.}

\received{\smonth{9} \syear{2012}}
\revised{\smonth{10} \syear{2013}}

%
\begin{abstract}
Motivated by the study of existence, uniqueness and regularity of
solutions to stochastic partial differential equations driven by jump
noise, we prove It\^{o} isomorphisms for $L^p$-valued stochastic
integrals with respect to a compensated Poisson random measure. The
principal ingredients for the proof are novel Rosenthal type
inequalities for independent random variables taking values in a
(noncommutative) $L^p$-space, which may be of independent interest. As
a by-product of our proof, we observe some moment estimates for the
operator norm of a sum of independent random matrices.
\end{abstract}

%


\end{frontmatter}

\section{Introduction}

In the functional analytic approaches to stochastic partial
differential equations (SPDEs), one studies an SPDE by reformulating it
as a stochastic ordinary differential equation in a suitable
infinite-dimensional state space $X$. A particularly popular method,
known as the semigroup approach, has proven very effective in obtaining
existence, uniqueness and regularity results for large classes of SPDEs
with Gaussian noise. A demonstration of this approach for SPDEs driven
by Gaussian noise in Hilbert spaces can be found in the monograph of Da
Prato and Zabczyk \cite{DaZ92}. In the last decade, there has been
increased interest in SPDEs driven by Poisson-type noise; see, for
instance, \cite{BrH09,FTP10,MPR10,MaR10} and the recent monograph
\cite
{PeZ07}. To obtain existence, uniqueness and regularity results for
such equations, one requires as a basic tool $L^p$-estimates for
vector-valued Poisson stochastic integrals. Concretely, one needs to
answer the following fundamental question. Suppose that we are given a
compensated Poisson random measure $\tilde{N}$ on $\R_+\times J$, where
$J$ is a $\sigma$-finite measure space, and a simple, adapted
$X$-valued process $F$. Can one find a suitable Banach space $\cI
_{p,X}$ such that
%
\begin{equation}
\label{eqn:ItoIsomorphismPoisson} c_{p,X} \tnorm{F}_{\cI_{p,X}} \leq \biggl(\E \biggl\|\int
_{\R_+\times
J} F \,d\tilde{N}\biggr \|_{X}^p
\biggr)^{{1}/{p}} \leq C_{p,X} \tnorm {F}_{\cI_{p,X}}
\end{equation}
for constants $c_{p,X},C_{p,X}$ depending only on $p$ and $X$? In the
SPDE literature, the right-hand side inequality is often referred to as
a Bichteler--Jacod inequality. This estimate allows one to define an
It\^{o}-type stochastic integral, sometimes called a strong or
$L^p$-stochastic integral in the literature \cite{App07,Rud04}, for all
elements in the closure of the simple adapted processes in $\cI_{p,X}$.
If both inequalities in (\ref{eqn:ItoIsomorphismPoisson}) hold
simultaneously, then we shall speak of an It\^{o} isomorphism. In this
situation, the choice of the space $\cI_{p,X}$ is optimal and therefore
$\cI_{p,X}$ provides the proper framework to study well-posedness and
regularity questions. We will call the corresponding Bichteler--Jacod
inequality optimal in this case, even though the constants
$c_{p,X},C_{p,X}$ in (\ref{eqn:ItoIsomorphismPoisson}) are not required
to be optimal. In the case of Gaussian noise, It\^{o} isomorphisms in
UMD Banach spaces were obtained in \cite{NVW07}. The optimality of
these estimates proved crucial in obtaining maximal regularity results
for stochastic parabolic evolution equations driven by Gaussian noise
\cite{NVW12}. One can expect optimal Bichteler--Jacod inequalities to
be similarly useful in the investigation of maximal regularity for
equations driven by Poisson or, more generally, L\'{e}vy noise.

Although Bichteler--Jacod inequalities are fundamental to the study of
SPDEs driven by jump noise and have been investigated by many authors
(see \cite{App07,BrH09,Hau11,MPR10,MaR10,PeZ07,Rud04} and the
references therein), a general It\^{o} isomorphism as available in the
Gaussian case is still missing. In fact, it seems that the optimality
of Bichteler--Jacod inequalities has not yet been investigated, not
even in the scalar-valued case. The main aim of this paper is to
provide optimal estimates of the form (\ref{eqn:ItoIsomorphismPoisson})
in the important case where $X$ is an $L^q$-space. On the one hand,
this result can serve as a stepping stone in the development of It\^{o}
isomorphisms in more general Banach spaces needed in the study of
SPDEs. On the other hand, our estimates are in itself valuable for
existence, uniqueness and regularity questions that can be addressed in
the setting of $L^q$-spaces; see, for example, \cite{MaR10} for interesting
examples.

With some additional effort, our estimates can be extended to the
situation where $X$ is a noncommutative $L^q$-space associated with a
semifinite von Neumann algebra $\cM$, for any $1<q<\infty$. To keep our
exposition accessible to readers who have little familiarity with
noncommutative analysis, we choose to focus on classical $L^q$-spaces
and only later indicate the modifications needed to prove our results
in full generality.

To formulate our main result for classical $L^q$-spaces, Theorem~\ref{thm1.1}, we introduce the following spaces. Let
$(S,\Si
,\sigma)$ be any measure space. We consider the completions $\cS
_q^p$, $\cD
_{q,q}^p$ and $\cD_{p,q}^p$ of the space of all simple functions in the
respective norms
%
\begin{eqnarray}
\label{eqn:normsSICom} \|F\|_{\cS_q^p} & =& \biggl(\E \biggl\| \biggl(\int
_{\R_+\ti J} |F|^2 \,dt\ti \,d\nu \biggr)^{{1}/{2}}
\biggr\|_{L^q(S)}^p \biggr)^{{1}/{p}},
\nonumber
\\
\|F\|_{\cD_{q,q}^p} & = &\biggl(\E \biggl(\int_{\R_+\ti J} \|F\|
_{L^q(S)}^q \,dt\ti \,d\nu \biggr)^{{p}/{q}}
\biggr)^{{1}/{p}},
\\
\|F\|_{\cD_{p,q}^p} & =& \biggl(\int_{\R_+\ti J} \E\|F
\|_{L^q(S)}^p \,dt\ti \,d\nu \biggr)^{{1}/{p}}.
\nonumber
\end{eqnarray}
We use the following notation. If $A,B$ are quantities depending on a
parameter $\alpha$, then we write $A\lesssim_{\alpha} B$ if there is a
constant $c_{\alpha}>0$ depending only on $\alpha$ such that $A\leq
c_{\alpha} B$. We write $A\simeq_{\alpha} B$ if both $A\lesssim
_{\alpha
} B$ and $B\lesssim_{\alpha} A$ hold. Also, we use $\chi_A$ to denote
the indicator function of a set $A$. Finally, to avoid ambiguity, let
us mention that we always take the notation $a<p,q<b$ to mean that both
$a<p<b$ and $a<q<b$ hold.
%
\begin{theorem}[(It\^{o} isomorphism)]\label{thm1.1}  Let $1<p,q<\infty
$. For any $B\in\cJ$, any $t>0$ and any simple, adapted $L^q(S)$-valued
process $F$,
%
\begin{equation}
\label{eqn:summarySILqPoisson} \biggl(\E\sup_{0<s\leq t}\biggl \|\int_{(0,s]\ti B}
F \,d\tilde{N}\biggr \|_{L^q(S)}^p \biggr)^{{1}/{p}} \simeq
_{p,q} \|F\chi_{(0,t]\ti
B}\|_{\cI_{p,q}},
\end{equation}
where $\cI_{p,q}$ is given by
\begin{eqnarray*}
\cS_{q}^p \cap\cD_{q,q}^p \cap
\cD_{p,q}^p &\qquad& \mathrm{if}\ 2\leq q\leq p<\infty,
\\
\cS_{q}^p \cap\bigl(\cD_{q,q}^p +
\cD_{p,q}^p\bigr) &\qquad& \mathrm{if}\ 2\leq p\leq q<\infty,
\\
\bigl(\cS_{q}^p \cap\cD_{q,q}^p
\bigr) + \cD_{p,q}^p&\qquad& \mathrm{if}\ 1<p<2\leq q<\infty,
\\
\bigl(\cS_{q}^p + \cD_{q,q}^p\bigr)
\cap\cD_{p,q}^p &\qquad& \mathrm{if}\ 1<q<2\leq p<\infty,
\\
\cS_{q}^p + \bigl(\cD_{q,q}^p \cap
\cD_{p,q}^p\bigr) &\qquad& \mathrm{if}\ 1<q\leq p\leq2,
\\
\cS_{q}^p + \cD_{q,q}^p +
\cD_{p,q}^p &\qquad&\mathrm{if}\ 1<p\leq q\leq2.
\end{eqnarray*}
Moreover, the estimate $\lesssim_{p,q}$ in (\ref
{eqn:summarySILqPoisson}) remains valid if $q=1$.
\end{theorem}
To understand the estimates in (\ref{eqn:summarySILqPoisson}), recall
that if $X$ and $Y$ are two Banach spaces which are continuously
embedded in some Hausdorff topological vector space, then their
intersection $X\cap Y$ and sum $X+Y$ are Banach spaces under the norms
\[
\|z\|_{X\cap Y} = \max\bigl\{\|z\|_{X},\|z\|_{Y}\bigr\}
\]
and
\[
\|z\|_{X+Y} = \inf\bigl\{\|x\|_{X} + \|y\|_{Y}\dvtx
z=x+y, x \in X, y \in Y\bigr\}.
\]
So, for example, if $2\leq p\leq q<\infty$ then $\|F\chi_{(0,t]\ti
B}\|
_{\cI_{p,q}}$ is equal to
\begin{eqnarray*}
&& \max \biggl[ \biggl(\E \biggl\| \biggl(\int_{(0,t]\ti B}
|F|^2 \,dt\ti \,d\nu \biggr)^{{1}/{2}} \biggr\|_{L^q(S)}^p
\biggr)^{{1}/{p}},
\\
&&\qquad\hspace*{3pt} \inf \biggl\{ \biggl(\E \biggl(\int_{(0,t]\ti B}
\|F_1\|_{L^q(S)}^q \,dt\ti \,d\nu \biggr)^{{p}/{q}}
\biggr)^{{1}/{p}}
\\
&&\hspace*{24pt}\qquad\quad{} + \biggl(\int_{(0,t]\ti B} \E\|F_2
\|_{L^q(S)}^p \,dt\ti \,d\nu \biggr)^{{1}/{p}} \biggr\}
\biggr],
\end{eqnarray*}
where the infimum is taken over all decompositions $F=F_1+F_2$ with
$F_1 \in\cD_{q,q}^p$ and $F_2 \in\cD_{p,q}^p$.

In comparison, recall that if $W$ is a Gaussian random measure on $\R
_+\ti J$, then for any $1<p,q<\infty$,
\[
\biggl(\E\sup_{0<s\leq t} \biggl\|\int_{(0,s]\ti B} F \,dW \biggr\|
_{L^q(S)}^p \biggr)^{{1}/{p}} \simeq_{p,q} \|F
\chi_{(0,t]\ti
B}\|_{\cS_q^p}.
\]
In the proof of the latter inequalities, as well as the more general
results in \cite{NVW07}, crucial use is made of the fact that any
mean-zero, real-valued Gaussian random variable has a standard normal
distribution once we divide it by its standard deviation. It is the
lack of this type of stability of Poisson random variables that
accounts for the more involved isomorphisms in Theorem~\ref{thm1.1}.

The result in Theorem~\ref{thm1.1} improves and extends
all the known estimates for $L^q$-valued Poisson stochastic integrals.
In fact, it seems that only the estimate ``$\lesssim_{p,q}$'' in (\ref
{eqn:summarySILqPoisson}) was obtained earlier in \cite{Hau11} for
$q=2$, $p=2^n$ for some $n \in\N$ (see also~\cite{MPR10} for a
near-optimal estimate for $q=2$, $2\leq p<\infty$). As it turns out,
this estimate is optimal. In all other cases, our optimal estimates
improve the results in the literature. We make a detailed comparison
with existing results at the end of Section~\ref
{sec:ItoPoissonNC}.

The proof of Theorem~\ref{thm1.1} relies on the
following decoupling inequalities. Let $\tilde{N}^c$ be a copy of
$\tilde{N}$ defined on a different probability space $(\Om_c,\cF
_c,\bP
_c)$, so that $\tilde{N}^c$ is independent of both $\tilde{N}$ and the simple,
adapted process $F$. If $X$ is a UMD Banach space, then for any
$1<p<\infty$,
%
\begin{equation}
\label{eqn:decouplingSIIntro} \biggl(\E \biggl\|\int_{(0,t]\ti B} F \,d\tilde{N}
\biggr\|_{X}^p \biggr)^{{1}/{p}} \simeq_{p,X}
\biggl(\E\E_c \biggl\|\int_{(0,t]\ti B} F \,d\tilde
{N}^c \biggr\|_{X}^p \biggr)^{{1}/{p}}.
\end{equation}
These inequalities are a special case of the decoupling inequalities
for martingale difference sequences in UMD Banach spaces due to
McConnell \cite{McC89} and Hitczenko \cite{HitUP}. A relatively simple
direct proof of (\ref{eqn:decouplingSIIntro}) can be found in, for
example \cite{Ver06}, Theorem 2.4.1. For completeness, we reproduce
this argument in Appendix~\ref{app:decoupling}. Observe that for a
simple process $F$, the decoupled stochastic integral on the right-hand
side can be written as a sum of conditionally independent, mean-zero
random variables. Thus, the key to obtaining an It\^{o} isomorphism as
in (\ref{eqn:ItoIsomorphismPoisson}) lies in answering the following
question: given $1\leq p<\infty$ and a Banach space~$X$, can we find
constants $c_{p,X},C_{p,X}$ depending only on $p$ and $X$ such that for
any sequence of independent, mean-zero $X$-valued random variables
$(\xi_i)$
%
\begin{equation}
\label{eqn:questionRos} c_{p,X} \bigl|\!\bigl|\!\bigl|(\xi_i)
\bigr|\!\bigr|\!\bigr|_{p,X} \leq
\biggl(\Ex\biggl \|\sum_i \xi_i
\biggr\|_X^p \biggr)^{{1}/{p}} \leq C_{p,X}
\bigl|\!\bigl|\!\bigl|(\xi_i)\bigr|\!\bigr|\!\bigr|_{p,X}
\end{equation}
for a suitable norm $\tnorm{\cdot}_{p,X}$ which can be computed
explicitly in terms of the (moments of the) individual summands $\xi
_i$? These kind of inequalities can be termed vector-valued Rosenthal
inequalities, since in the case $X=\C$ the well-known answer to this
question is due to Rosenthal \cite{Ros70}: For $2\leq p<\infty$, there
exists an absolute constant $c$ such that
%
\begin{eqnarray}
\label{eqn:ClassicalRosIntro} \qquad\biggl(\E \biggl|\sum_i
\xi_i \biggr|^p \biggr)^{{1}/{p}} & \leq& c
\frac{p}{\log p} \max \biggl\{ \biggl(\sum_i \E|
\xi_i|^p \biggr)^{{1}/{p}}, \biggl(\sum
_i \E|\xi_i|^2
\biggr)^{{1}/{2}} \biggr\},
\nonumber
\\[-8pt]
\\[-8pt]
\nonumber
\biggl(\E \biggl|\sum_i \xi_i
\biggr|^p \biggr)^{{1}/{p}} & \geq&\frac{1}{2} \max \biggl\{
\biggl(\sum_i \E|\xi_i|^p
\biggr)^{{1}/{p}}, \biggl(\sum_i \E|\xi
_i|^2 \biggr)^{{1}/{2}} \biggr\}.
\end{eqnarray}
A version of (\ref{eqn:ClassicalRosIntro}) for noncommutative random
variables, as well as a version for $1<p\leq2$, was recently obtained
by Junge and Xu \cite{JuX08}. Their main results yield two-sided bounds
of the form (\ref{eqn:questionRos}) if $X$ is a (noncommutative)
$L^q$-space and $p=q$. Various upper bounds for the moments of a
martingale with values in a uniformly $2$-smooth Banach space were
obtained by Pinelis \cite{Pin94}. However, these results lead to a
two-sided estimate of the form (\ref{eqn:questionRos}) only if $X$ is a
Hilbert space (see \cite{Pin94}, Theorem 5.2).

Our main result in this direction provides Rosenthal-type inequalities
for independent random variables taking values in a noncommutative
$L^q$-space. We state the version for classical $L^q$-spaces. We
consider the following norms on the linear space of all finite
sequences $(f_i)$ of random variables in $L^{\infty}(\Om;L^q(S))$. For
$1\leq p,q<\infty$, we set
%
\begin{eqnarray}
\label{eqn:normsRosCom} \bigl\|(f_i)\bigr\|_{S_{q}} & =& \biggl\| \biggl(\sum
_i \E|f_i|^2 \biggr)^{{1}/{2}}
\biggr\|_{L^q(S)},
\nonumber
\\[-8pt]
\\[-8pt]
\nonumber
\bigl\|(f_i)\bigr\|_{D_{p,q}} & =& \biggl(\sum
_i\E\|f_i\|_{L^q(S)}^p
\biggr)^{{1}/{p}}.
\end{eqnarray}
%
\begin{theorem}
\label{thm:summaryRosIntro} Let $1<p,q<\infty$ and let $(S,\Si
,\sigma)$ be
a measure space. If $(\xi_i)$ is a sequence of independent, mean-zero
random variables taking values in $L^q(S)$, then
%
\begin{equation}
\label{eqn:summaryRos} \biggl(\E \biggl\|\sum_i
\xi_i \biggr\| _{L^q(S)}^p \biggr)^{{1}/{p}}
\simeq_{p,q} \bigl\|(\xi_i)\bigr\|_{s_{p,q}},
\end{equation}
where $s_{p,q}$ is given by
\begin{eqnarray*}
S_{q} \cap D_{q,q} \cap D_{p,q} &\qquad& \mathrm{if}\ 2\leq
q\leq p<\infty,
\\
S_{q} \cap(D_{q,q} + D_{p,q}) &\qquad& \mathrm{if}\ 2\leq p
\leq q<\infty,
\\
(S_{q} \cap D_{q,q}) + D_{p,q} &\qquad&\mathrm{if}\ 1<p<2
\leq q<\infty,
\\
(S_{q} + D_{q,q}) \cap D_{p,q} &\qquad& \mathrm{if}\ 1<q<2
\leq p<\infty,
\\
S_{q} + (D_{q,q} \cap D_{p,q}) &\qquad& \mathrm{if}\ 1<q
\leq p\leq2,
\\
S_{q} + D_{q,q} + D_{p,q} &\qquad& \mathrm{if}\ 1<p\leq q
\leq2.
\end{eqnarray*}
Moreover, the estimate $\lesssim_{p,q}$ in (\ref{eqn:summaryRos})
remains valid if $p=1$, $q=1$ or both.
\end{theorem}
The notational similarity between the spaces introduced in (\ref
{eqn:normsSICom}) and (\ref{eqn:normsRosCom}) is intentional. Indeed,
when applying Theorem~\ref{thm:summaryRosIntro} to the decoupled
Poisson stochastic integral on the right-hand side of (\ref
{eqn:decouplingSIIntro}), the spaces $S_q$, $D_{q,q}$, and $D_{p,q}$
give rise to $\cS^p_q$, $\cD^p_{q,q}$ and $\cD^p_{p,q}$,
respectively.

If $p=q$, then the result in Theorem~\ref{thm:summaryRosIntro} (as well
as its generalization in Theorem~\ref{thm:summaryRosenthalLqNC}) is a
special case of the noncommutative Rosenthal inequalities in \cite
{JuX08} and the only novelty here is a new proof. However, in
applications of Theorem~\ref{thm1.1}, and hence of
Theorem~\ref{thm:summaryRosIntro}, one is typically also interested in
the case $p\neq q$.

As said before, we can even prove an extension of the It\^{o}
isomorphism in Theorem~\ref{thm1.1} in which $L^q(S)$
is replaced by a general noncommutative $L^q$-space associated with a
semifinite von Neumann algebra $\cM$. This result is stated and proved
in Theorem~\ref{thm:summarySILqPoissonNC} below. The proof proceeds
along the same lines as the result for classical $L^q$-spaces and in
particular requires a version of the Rosenthal-type inequalities stated
above for random variables taking values in a noncommutative
$L^q$-space, which we prove in Theorem~\ref{thm:summaryRosenthalLqNC}.
As a by-product of the proof of Theorem~\ref{thm:summaryRosenthalLqNC},
we take the opportunity to observe the following estimates for the
moments of the operator norm of a sum of independent, mean-zero $d_1\ti
d_2$ random matrices $(x_i)$, which may be of independent interest. If
$2\leq p<\infty$ and $d=\min\{d_1,d_2\}$, then
\begin{eqnarray*}
\biggl(\Ex \biggl\|\sum_i x_i
\biggr\|^p \biggr)^{{1}/{p}} & \leq& C_{p,d} \max \biggl\{ \biggl\|
\biggl(\sum_i \Ex|x_i|^2
\biggr)^{{1}/{2}}\biggr \|,\biggl \| \biggl(\sum_i
\Ex\bigl|x_i^*\bigr|^2 \biggr)^{{1}/{2}} \biggr\|,\\
&&\hspace*{120pt} C_{{p}/{2},d} \Bigl(\Ex\max_i \|x_i
\|^p \Bigr)^{{1}/{p}} \biggr\},
\end{eqnarray*}
where $C_{p,d}$ is of order $\max\{\sqrt{p},\sqrt{\log d}\}$. In
Section~\ref{sec:randomMatrices}, we compare this result to known
estimates for random matrices.

An application of Theorem~\ref{thm1.1} is discussed in
\cite{DMN12}.

\section{$L^q$-valued Rosenthal inequalities}

We start by proving Theorem~\ref{thm:summaryRosIntro}. Throughout, we
fix a measure space $(S,\Si,\sigma)$. Let us collect some tools that we
will use in the proof. First recall the Khintchine inequalities for
$L^q(S)$. Let $(r_i)$ be a Rademacher sequence, that is, a sequence of
independent, identically distributed random variables satisfying $\bP
(r_i=1)=\bP(r_i=-1)=1/2$. Then, for any $0<p,q<\infty$ and any finite
sequence $(x_i)$ in $L^q(S)$ we have
%
\begin{equation}
\label{eqn:KhiComLp} \biggl(\E \biggl\|\sum_{i}
r_i x_i \biggr\|_{L^q(S)}^p
\biggr)^{{1}/{p}} \simeq _{p,q}\biggl \| \biggl(\sum
_{i}|x_i|^2 \biggr)^{{1}/{2}}
\biggr\|_{L^q(S)}.
\end{equation}
We will frequently use this result in combination with the following
well-known symmetrization inequalities (see, e.g., \cite{LeT91}, Lemma
6.3). Let $1\leq p<\infty$, let $X$ be a Banach space and $(\xi_i)$ a
sequence of independent, mean-zero $X$-valued random variables. If
$(r_i)$ is a Rademacher sequence defined on a probability space $(\Om
_r,\cF_r,\bP_r)$, then
%
\begin{equation}
\label{eqn:symmetrization} \qquad\frac{1}{2} \biggl(\E \biggl\|\sum_i
\xi_i \biggr\|_{X}^p \biggr)^{{1}/{p}} \leq
\biggl(\E_r\E \biggl\|\sum_i
r_i \xi_i \biggr\|_{X}^p
\biggr)^{{1}/{p}} \leq2 \biggl(\E \biggl\|\sum_i
\xi_i \biggr\|_{X}^p \biggr)^{{1}/{p}}.
\end{equation}
As a first consequence, we find the following useful estimates.
%
\begin{lemma}
\label{lem:squareFunEst2}
Suppose that $1\leq p,q\leq2$. Let $(\xi_i)$ be a finite sequence of
independent, mean-zero $L^q(S)$-valued random variables. Then
\[
\biggl(\E \biggl\|\sum_i \xi_i
\biggr\|_{L^q(S)}^p \biggr)^{{1}/{p}} \lesssim _{p,q} \biggl\|
\biggl(\sum_i \E|\xi_i|^2
\biggr)^{{1}/{2}} \biggr\|_{L^q(S)}.
\]
On the other hand, if $2\leq p,q<\infty$ then
\[
\biggl\| \biggl(\sum_i \E|\xi_i|^2
\biggr)^{{1}/{2}} \biggr\|_{L^q(S)} \lesssim _{p,q} \biggl(\E \biggl\|
\sum_i \xi_i\biggr \|_{L^q(S)}^p
\biggr)^{{1}/{p}}.
\]
\end{lemma}
\begin{pf}
Let $1\leq p,q\leq2$. Combining (\ref{eqn:symmetrization}) and (\ref
{eqn:KhiComLp}) yields
\begin{eqnarray*}
\biggl(\E \biggl\|\sum_i \xi_i
\biggr\|_{L^q(S)}^p \biggr)^{{1}/{p}} & \simeq& \hspace*{-2pt}_{p,q}
\biggl(\E \biggl\| \biggl(\sum_i |\xi_i|^2
\biggr)^{1/2} \biggr\| _{L^q(S)}^p \biggr)^{{1}/{p}}
\\
& = &\biggl(\E \biggl\|\sum_i |\xi_i|^2
\biggr\|_{L^{q/2}(S)}^{p/2} \biggr)^{{1}/{p}}
\\
& \leq& \biggl\|\sum_i \E|\xi_i|^2
\biggr\|_{L^{{q}/{2}}(S)}^{1/2} = \biggl\| \biggl(\sum_i
\E|\xi_i|^2 \biggr)^{{1}/{2}} \biggr\| _{L^q(S)}.
\end{eqnarray*}
Note that in the final inequality we apply Jensen's
inequality, using that $\frac{p}{2},\frac{q}{2}<1$. If we assume
$2\leq
p,q<\infty$, then this inequality is reversed.
\end{pf}
We recall the notions of type and cotype. A Banach space $X$ is said to
have \emph{type~$s$} for some $1\leq s\leq2$ if for any finite
sequence $(x_i)$ in $X$
\[
\biggl(\Ex \biggl\|\sum_i r_i
x_i \biggr\|_X^2 \biggr)^{{1}/{2}} \lesssim
_{s,X} \biggl(\sum_i\|x_i
\|_X^s \biggr)^{{1}/{s}}.
\]
A Banach space $X$ is said to have \emph{cotype $s$} for some $2\leq
s<\infty$ if for any finite sequence $(x_i)$ in $X$
\[
\biggl(\sum_i\|x_i
\|_X^s \biggr)^{{1}/{s}} \lesssim_{s,X}
\biggl(\Ex \biggl\| \sum_i r_i
x_i \biggr\|_X^2 \biggr)^{{1}/{2}}.
\]
It is well known that any $L^q$-space with $1\leq q<\infty$ has type
$\min\{q,2\}$ and cotype $\max\{q,2\}$. The following observation is
well known, we include a proof for the convenience of the reader. The
main ingredients are Kahane's inequalities (see, e.g., \cite{LeT91},
Theorem 4.7): for any $0<p,q<\infty$ there exists a constant $\kappa
_{p,q}$ such that for any Banach space $X$ and $x_1,\ldots,x_n \in X$,
%
\begin{equation}
\label{eqn:Kahane} \Biggl(\E \Biggl\|\sum_{i=1}^n
r_i x_i \Biggr\|_X^p
\Biggr)^{{1}/{p}} \leq \kappa_{p,q} \Biggl(\E \Biggl\|\sum
_{i=1}^n r_i x_i
\Biggr\|_X^q \Biggr)^{{1}/{q}}.
\end{equation}
%
\begin{lemma}
\label{lem:typeCotypeEst}
Fix $1\leq p<\infty$. Let $X$ be a Banach space and $(\xi_i)$ be a
finite sequence of independent, mean-zero $X$-valued random variables.
If $X$ has type $1\leq s\leq2$, then
\[
\biggl(\Ex \biggl\|\sum_i \xi_i
\biggr\|_X^p \biggr)^{{1}/{p}} \lesssim _{p,s,X}
\biggl(\Ex \biggl(\sum_i \|\xi_i
\|_X^s \biggr)^{p/s} \biggr)^{{1}/{p}}.
\]
On the other hand, if $X$ has cotype $2\leq s<\infty$, then
\[
\biggl(\Ex \biggl(\sum_i \|\xi_i
\|_X^s \biggr)^{p/s} \biggr)^{{1}/{p}}
\lesssim_{p,s,X} \biggl(\Ex \biggl\|\sum_i
\xi_i \biggr\|_X^p \biggr)^{{1}/{p}}.
\]
\end{lemma}
\begin{pf}
Suppose $X$ has type $s$. By symmetrization, Kahane's inequalities and
the type $s$ inequality we obtain
\begin{eqnarray*}
\biggl(\Ex \biggl\|\sum_i \xi_i
\biggr\|_X^p \biggr)^{{1}/{p}} & \simeq& \biggl(\Ex
\Ex_r \biggl\|\sum_i r_i
\xi_i \biggr\|_X^p \biggr)^{{1}/{p}}
\nonumber
\\
& \simeq&\hspace*{-4pt}_{p} \biggl(\Ex \biggl(\Ex_r \biggl\|\sum
_i r_i \xi_i \biggr\|
_X^2 \biggr)^{p/2} \biggr)^{{1}/{p}}
\lesssim_{s,X} \biggl(\Ex \biggl(\sum_i
\| \xi_i\|_X^s \biggr)^{p/s}
\biggr)^{{1}/{p}}.
\end{eqnarray*}
The second assertion is proved similarly.
\end{pf}
The following result is the key to the Rosenthal-type inequalities in
the cases where $2\leq p,q<\infty$.
%
\begin{theorem}
\label{thm:2pqqp} Suppose that $2\leq p,q<\infty$. If $(\xi_i)$ is a
finite sequence of independent, mean-zero $L^q(S)$-valued random
variables, then
%
\begin{eqnarray}
\label{eqn:2pqqp} && \biggl(\E \biggl\|\sum_i
\xi_i \biggr\|_{L^q(S)}^p \biggr)^{{1}/{p}}
\nonumber
\\[-8pt]
\\[-8pt]
\nonumber
&&\qquad \simeq_{p,q} \max \biggl\{ \biggl\| \biggl(\sum
_i \E|\xi_i|^2
\biggr)^{1/2} \biggr\|_{L^q(S)}, \biggl(\E \biggl(\sum
_i \|\xi_i\| _{L^q(S)}^q
\biggr)^{{p}/{q}} \biggr)^{{1}/{p}} \biggr\}.
\end{eqnarray}
\end{theorem}
\begin{pf}
We first prove the estimate $\gtrsim_{p,q}$. By Lemma~\ref{lem:squareFunEst2},
\[
\biggl\| \biggl(\sum_i \E|\xi_i|^2
\biggr)^{{1}/{2}} \biggr\|_{L^q(S)} \lesssim _{p,q} \biggl(\E \biggl\|
\sum_i \xi_i \biggr\|_{L^q(S)}^p
\biggr)^{{1}/{p}}.
\]
Moreover, since $L^q(S)$ has cotype $q$ Lemma~\ref{lem:typeCotypeEst} implies
\begin{eqnarray*}
\biggl(\E \biggl(\sum_i \|\xi_i
\|_{L^q(S)}^q \biggr)^{{p}/{q}} \biggr)^{{1}/{p}} &
\lesssim_{p,q} \biggl(\E \biggl\|\sum_i
\xi_i \biggr\| _{L^q(S)}^p \biggr)^{{1}/{p}}.
\end{eqnarray*}
We now prove the reverse inequality in (\ref{eqn:2pqqp}). By
symmetrization and the Khintchine inequalities (\ref{eqn:KhiComLp}),
%
\begin{equation}
\label{eqn:estIns} \biggl(\E \biggl\|\sum_i
\xi_i \biggr\|_{L^q(S)}^p \biggr)^{{1}/{p}} \simeq
_{p,q} \biggl(\E \biggl\| \biggl(\sum_i |
\xi_i|^2 \biggr)^{1/2}\biggr \| _{L^q(S)}^p
\biggr)^{{1}/{p}}.
\end{equation}
By the triangle inequality, we obtain
%
\begin{eqnarray}
\label{eqn:simpleEstFors} && \biggl(\E \biggl\| \biggl(\sum_i |
\xi_i|^2 \biggr)^{{1}/{2}}\biggr \| _{L^q(S)}^p
\biggr)^{{1}/{p}}
\nonumber
\\
&&\qquad = \biggl(\E \biggl\|\sum_i |\xi_i|^2
\biggr\|_{L^{q/2}(S)}^{p/2} \biggr)^{{1}/{p}}
\\
&&\qquad \leq \biggl( \biggl(\E\biggl \|\sum_i |
\xi_i|^2 - \E|\xi_i|^2 \biggr\|
_{L^{q/2}(S)}^{p/2} \biggr)^{2/p} + \biggl\|\sum
_i \E |\xi _i|^2
\biggr\|_{L^{{q}/{2}}(S)} \biggr)^{{1}/{2}}.\nonumber
\end{eqnarray}
Suppose first that $q\leq4$. Then $L^{{q}/{2}}(S)$ has type
$\frac
{q}{2}$, so by Lemma~\ref{lem:typeCotypeEst},
\begin{eqnarray*}
&& \biggl(\E \biggl\|\sum_i |\xi_i|^2
- \E|\xi_i|^2 \biggr\|_{L^{q/2}(S)}^{p/2}
\biggr)^{2/p}
\\
& &\qquad\lesssim_{p,q} \biggl(\E \biggl(\sum_i
\bigl\| |\xi_i|^2 - \E|\xi _i|^2\bigr\|
_{L^{{q}/{2}}(S)}^{{q}/{2}} \biggr)^{{p}/{q}} \biggr)^{2/p}
\\
&&\qquad \leq \biggl(\E \biggl(\sum_i \|
\xi_i\|_{L^q(S)}^{q} \biggr)^{p/q}
\biggr)^{2/p} + \biggl(\sum_i \bigl\|\E|
\xi_i|^2\bigr\|_{L^{q/2}(S)}^{{q}/{2}}
\biggr)^{2/q}
\\
&&\qquad \leq \biggl(\E \biggl(\sum_i \|
\xi_i\|_{L^q(S)}^{q} \biggr)^{p/q}
\biggr)^{2/p} + \E \biggl(\sum_i \bigl\| |
\xi_i|^2\bigr\|_{L^{q/2}(S)}^{{q}/{2}}
\biggr)^{2/q}
\\
& &\qquad\leq2 \biggl(\E \biggl(\sum_i \|
\xi_i\|_{L^q(S)}^{q} \biggr)^{p/q}
\biggr)^{2/p},
\end{eqnarray*}
where in the final two steps we apply Jensen's inequality, using that
the $\ell^{{q}/{2}}(L^{{q}/{2}}(S))$-norm is convex, and
subsequently use H\"{o}lder's inequality.

Suppose now that $q>4$. By applying symmetrization and the Khintchine
inequalities (\ref{eqn:KhiComLp}), we find
%
\begin{eqnarray}
\label{eqn:AuxEstq2sNo1} & &\biggl(\E \biggl\|\sum_i |
\xi_i|^2 - \E|\xi_i|^2
\biggr\|_{L^{q/2}(S)}^{p/2} \biggr)^{2/p}
\nonumber
\\
& &\qquad\simeq_{p,q} \biggl(\E \biggl\| \biggl(\sum_i
\bigl| |\xi_i|^2 - \E|\xi _i|^2\bigr|^2
\biggr)^{{1}/{2}} \biggr\|_{L^{{q}/{2}}(S)}^{p/2} \biggr)^{2/p}
\nonumber
\\[-8pt]
\\[-8pt]
\nonumber
&&\qquad \leq \biggl(\E \biggl\| \biggl(\sum_i |
\xi_i|^{4} \biggr)^{1/2} \biggr\| _{L^{{q}/{2}}(S)}^{p/2}
\biggr)^{2/p} + \biggl\| \biggl(\sum_i \bigl|\E|
\xi_i|^2\bigr|^2 \biggr)^{{1}/{2}}
\biggr\|_{L^{{q}/{2}}(S)}
\\
& &\qquad\leq \biggl(\E \biggl\| \biggl(\sum_i |
\xi_i|^{4} \biggr)^{1/4} \biggr\| _{L^{q}(S)}^{p}
\biggr)^{2/p} +\biggl \|\sum_i \E|
\xi_i|^2 \biggr\| _{L^{{q}/{2}}(S)}.\nonumber
\end{eqnarray}
Since $q>4$, there is some $0<\theta<\frac{1}{2}$ such that $\frac
{1}{4} = \frac{\theta}{2} + \frac{1-\theta}{q}$. By applying H\"
{o}lder's inequality three times (the second and third time with
parameters $\frac{1}{q}=\frac{\theta}{q} + \frac{1-\theta}{q}$ and
$\frac{1}{p}=\frac{\theta}{p} + \frac{1-\theta}{p}$, resp.),
we obtain
%
\begin{eqnarray}
\label{eqn:AuxEstq2sNo2} && \biggl(\E \biggl\| \biggl(\sum_i |
\xi_i|^{4} \biggr)^{1/4}\biggr \| _{L^{q}(S)}^{p}
\biggr)^{2/p}
\nonumber
\\
&&\qquad \leq \biggl(\E \biggl\| \biggl(\sum_i |
\xi_i|^{2} \biggr)^{{\theta}/{2}} \biggl(\sum
_i |\xi_i|^{q}
\biggr)^{{(1-\theta)}/{q}} \biggr\| _{L^{q}(S)}^{p} \biggr)^{2/p}
\nonumber
\\
&&\qquad \leq \biggl(\E \biggl( \biggl\| \biggl(\sum_i |
\xi_i|^{2} \biggr)^{{\theta}/{2}}\biggr \|_{L^{{q}/{\theta}}(S)} \biggl\|
\biggl(\sum_i |\xi _i|^{q}
\biggr)^{{(1-\theta)}/{q}} \biggr\|_{L^{{q}/{(1-\theta)}}(S)} \biggr)^{p}
\biggr)^{2/p}
\nonumber
\\[-8pt]
\\[-8pt]
\nonumber
&&\qquad \leq \biggl(\E \biggl\| \biggl(\sum_i |
\xi_i|^{2} \biggr)^{{\theta}/{2}} \biggr\|_{L^{{q}/{\theta}}(S)}^{{p}/{\theta}}
\biggr)^{{2\theta}/{p}}\\
&&\hspace*{36pt}{}\times \biggl(\E \biggl\| \biggl(\sum_i
|\xi_i|^{q} \biggr)^{{(1-\theta)}/{q}}\biggr \|_{L^{{q}/{(1-\theta)}}(S)}^
{{p}/{(1-\theta)}}
\biggr)^{{2(1-\theta)}/{p}}
\nonumber
\\
&&\qquad = \biggl(\E\biggl \| \biggl(\sum_i |
\xi_i|^{2} \biggr)^{{1}/{2}} \biggr\| _{L^{q}(S)}^{p}
\biggr)^{{2\theta}/{p}} \biggl(\E \biggl\| \biggl(\sum_i
|\xi _i|^{q} \biggr)^{{1}/{q}} \biggr\|_{L^{q}(S)}^{p}
\biggr)^{{2(1-\theta)}/{p}}.\nonumber
\end{eqnarray}
Combining (\ref{eqn:simpleEstFors}), (\ref{eqn:AuxEstq2sNo1}) and
(\ref
{eqn:AuxEstq2sNo2}), we arrive at the inequality
\[
a^2 \lesssim_{p,q} a^{2\theta}b^{2(1-\theta)} +
c^2,
\]
where we set $a = (\E\|(\sum_i |\xi_i|^2)^{{1}/{2}}\|
_{L^q(S)}^p)^{{1}/{p}}$, $b=(\Ex(\sum_i \|\xi_i\|
_{L^q(S)}^q)^{p/q})^{{1}/{p}}$ and $c=\|(\sum_i \E|\xi_i|^2)
^{1/2}\|
_{L^q(S)}$. Notice that if $a\leq b$ then the claim immediately follows
from (\ref{eqn:estIns}). Hence, we may assume $a>b$. Since $0<2\theta
<1$ we then have
\[
a^{2\theta}b^{2(1-\theta)} = b^2 \biggl(\frac{a}{b}
\biggr)^{2\theta} \leq ab.
\]
Thus, we obtain the inequality
\[
a^2 \lesssim_{p,q} ab + c^2.
\]
Solving this quadratic inequality, we find that $a\lesssim_{p,q} \max
\{
b,c\}$. That is,
\begin{eqnarray*}
&& \biggl(\E \biggl\| \biggl(\sum_i |
\xi_i|^2 \biggr)^{{1}/{2}}\biggr \| _{L^q(S)}^p
\biggr)^{{1}/{p}}
\\
&&\qquad \lesssim_{p,q} \max \biggl\{ \biggl(\Ex \biggl(\sum
_i \|\xi_i\| _{L^q(S)}^q
\biggr)^{{p}/{q}} \biggr)^{{1}/{p}},\biggl \| \biggl(\sum
_i \E|\xi _i|^2
\biggr)^{{1}/{2}}\biggr \|_{L^q(S)} \biggr\}.
\end{eqnarray*}
The result now follows from (\ref{eqn:estIns}). This completes the proof.
\epf
\end{pf}
Recall the spaces $s_{p,q}$ defined in the statement of Theorem~\ref
{thm:summaryRosIntro}. In the proof of this result, we shall make use
of the fact that for any $1<p,q<\infty$
%
\begin{equation}
\label{eqn:dualspq} (s_{p,q})^* = s_{p',q'},\qquad \biggl(
\frac{1}{p} + \frac{1}{p'} = 1, \frac
{1}{q} +
\frac{1}{q'} = 1\biggr)
\end{equation}
holds isometrically. This follows from the following general principle.
Suppose that $X$ and $Y$ are two Banach spaces which are continuously
embedded in some Hausdorff topological vector space and assume moreover
that $X\cap Y$ is dense in both $X$ and $Y$. Then we have
%
\begin{equation}
\label{eqn:sumIntersectionDuality} (X\cap Y)^* = X^* + Y^*,\qquad (X+Y)^* = X^*\cap Y^*
\end{equation}
isometrically. The duality brackets under these identifications are
given by
\[
\bigl\langle x,x^*\bigr\rangle= \bigl\langle x,x^*|_{X\cap Y}\bigr\rangle\qquad
\bigl(x^* \in X^* + Y^*\bigr),
\]
where $x^*|_{X\cap Y}$ denotes the restriction of $x^*$ to $X\cap Y$, and
\[
\bigl\langle x, x^*\bigr\rangle= \bigl\langle y,x^*\bigr\rangle+ \bigl\langle
z,x^*\bigr\rangle\qquad \bigl(x^* \in X^*\cap Y^*, x=y+z \in X+Y\bigr),
\]
respectively; see, for example, \cite{KPS82}, Theorem I.3.1. In our case of
interest, the spaces $S_{q}$, $D_{p,q}$ and $D_{q,q}$ have dense
intersection and, therefore, the duality of these individual spaces
imply together with (\ref{eqn:sumIntersectionDuality}) that (\ref
{eqn:dualspq}) holds, with associated duality bracket
\[
\bigl\langle(f_i),(g_i)\bigr\rangle= \sum
_i \E\int f_i g_i \,d\sigma.
\]
We need two more ingredients for the proof of Theorem~\ref
{thm:summaryRosIntro}. The first are the hypercontractive-type
inequalities due to Hoffmann--J{\o}rgensen \cite{Hof74} (see also
\cite
{KwW92,LeT91} for a proof yielding a constant of optimal order)
%
\begin{equation}
\label{eqn:HJTKSIntro} \qquad\biggl(\E \biggl\|\sum_i
\xi_i \biggr\|_X^p \biggr)^{{1}/{p}} \lesssim
\frac{p}{\log2p} \biggl(\E \biggl\|\sum_i
\xi_i \biggr\|_X + \Bigl(\E\max_i \|
\xi_i\|_X^p \Bigr)^{{1}/{p}} \biggr),
\end{equation}
valid for any $1\leq p<\infty$ and any sequence $(\xi_i)$ of
independent, mean-zero random variables taking values in a Banach space
$X$. Finally, let us recall the Rosenthal inequalities for a sequence
$(f_i)$ of positive scalar-valued random variables: if $1\leq p<\infty
$, then
%
\begin{equation}
\label{eqn:RosPos} \biggl(\E\biggl |\sum_{i}
f_i \biggr|^p \biggr)^{{1}/{p}} \lesssim_p
\max \biggl\{ \biggl(\sum_i\E|f_i|^p
\biggr)^{{1}/{p}}, \sum_i
\E|f_i| \biggr\}.
\end{equation}
We are now ready to prove our first main result.
\begin{pf*}{Proof of Theorem~\ref{thm:summaryRosIntro}}
Let us note that the inequalities ``$\gtrsim_{p,q}$'' in (\ref
{eqn:summaryRos}) follow by duality once the reverse inequalities have
been established. Indeed, if $(\eta_i)$ is a finite sequence in
$s_{p',q'}$ of norm $1$, then
%
\begin{eqnarray}
\label{eqn:dualArgInd1} \bigl\langle(\xi_i),(\eta_i)\bigr
\rangle& =& \sum_i \E\int(\xi_i\eta
_i)\,d\sigma
\nonumber
\\
& =& \sum_i \E\int\bigl(\xi_i \bigl(
\E(\eta_i|\xi_i) - \E(\eta _i)\bigr)\bigr)\,d
\sigma
\nonumber
\\
& =& \sum_{i,j} \E\int\bigl(\xi_i\bigl(
\E(\eta_j|\xi_j) - \E(\eta _j)\bigr)\bigr)\,d
\sigma
\\
& =& \E\int \biggl(\sum_i\xi_i
\biggr) \biggl(\sum_j\E(\eta_j|
\xi_j) - \E(\eta _j) \biggr)\,d\sigma
\nonumber
\\
& \leq &\biggl(\E \biggl\|\sum_i \xi_i
\biggr\|_{L^q(S)}^p \biggr)^{{1}/{p}} \biggl(\E\biggl \|\sum
_j \E(\eta_j|\xi_j) - \E(
\eta_j) \biggr\| _{L^{q'}(S)}^{p'} \biggr)^{{1}/{p'}}.\nonumber
\end{eqnarray}
Since the elements $\E(\eta_j|\xi_j) - \E(\eta_j)$ are independent and
mean-zero,
%
\begin{eqnarray}
\label{eqn:dualArgInd2} \bigl\langle(\xi_i),(\eta_i)\bigr
\rangle & \lesssim&\hspace*{-2pt}_{p',q'} \biggl(\E \biggl\|\sum_i
\xi_i \biggr\| _{L^q(S)}^p \biggr)^{{1}/{p}} \bigl\|
\bigl(\E(\eta_j|\xi_j)-\E(\eta_j)\bigr)
\bigr\|_{s_{p',q'}}
\nonumber
\\[-8pt]
\\[-8pt]
\nonumber
& \leq&2 \biggl(\E \biggl\|\sum_i \xi_i
\biggr\|_{L^q(S)}^p \biggr)^{{1}/{p}}.
\end{eqnarray}
By (\ref{eqn:dualspq}), the claim follows by taking the supremum over
all $(\eta_i)$ as above. We now prove the estimates $\lesssim_{p,q}$
case by case.

\emph{Case $2\leq q\leq p<\infty$}: Recall that Theorem~\ref{thm:2pqqp}
says that
\begin{eqnarray*}
&& \biggl(\E \biggl\|\sum_i \xi_i
\biggr\|_{L^q(S)}^p \biggr)^{{1}/{p}}
\\
&&\qquad \lesssim_{p,q} \max \biggl\{ \biggl\| \biggl(\sum
_i \E|\xi_i|^2
\biggr)^{1/2} \biggr\|_{L^q(S)}, \biggl(\E \biggl(\sum
_i \|\xi_i\| _{L^q(S)}^q
\biggr)^{{p}/{q}} \biggr)^{{1}/{p}} \biggr\}.
\end{eqnarray*}
Since $q\leq p$, applying (\ref{eqn:RosPos}) with $f_i = \|\xi\|
_{L^q(S)}^q$ yields
\begin{eqnarray*}
&&\biggl(\E \biggl(\sum_i \|\xi_i
\|_{L^q(S)}^q \biggr)^{{p}/{q}} \biggr)^{{1}/{p}}
\nonumber
\\[-8pt]
\\[-8pt]
\nonumber
&&\qquad\lesssim_{p,q} \max \biggl\{ \biggl(\sum_i
\Ex\|\xi_i\| _{L^q(S)}^p \biggr)^{{1}/{p}},
\biggl(\sum_i \Ex\|\xi_i
\|_{L^q(S)}^q \biggr)^{1/q} \biggr\}.
\end{eqnarray*}

\emph{Case $2\leq p\leq q<\infty$}: If $p\leq q$, the contractive
embeddings $L^q(\Om)\subset L^p(\Om)$ and $\ell^p\subset\ell^q$ imply
%
\begin{equation}
\label{eqn:littlesInf0} \biggl(\Ex \biggl(\sum_i \|
\xi_i\|_{L^q(S)}^q \biggr)^{{p}/{q}}
\biggr)^{{1}/{p}} \leq \biggl(\sum_i \Ex\|
\xi_i\|_{L^q(S)}^q \biggr)^{{1}/{q}}
\end{equation}
and
%
\begin{equation}
\label{eqn:littlesInf1} \biggl(\Ex \biggl(\sum_i \|
\xi_i\|_{L^q(S)}^q \biggr)^{{p}/{q}}
\biggr)^{{1}/{p}} \leq \biggl(\sum_i \Ex\|
\xi_i\|_{L^q(S)}^p \biggr)^{{1}/{p}}.
\end{equation}
By the triangle inequality,
\[
\biggl(\Ex \biggl(\sum_i \|\xi_i
\|_{L^q(S)}^q \biggr)^{{p}/{q}} \biggr)^{{1}/{p}} \leq
\bigl\|(\xi_i)\bigr\|_{D_{p,q}+D_{q,q}}.
\]
The asserted estimate now follows from Theorem~\ref{thm:2pqqp}.

\emph{Case $1\leq p\leq q\leq2$}: Let $(\eta_{i}) \in S_q$, $(\theta
_i) \in D_{p,q}$ and $(\kappa_i) \in D_{q,q}$ be
such that $\xi_i = \eta_{i} + \theta_i + \kappa_i$. Then
\[
\xi_i = \E(\eta_{i}|\xi_i) - \E(
\eta_{i}) + \E(\theta_i|\xi_i) - \E (
\theta_i) + \E(\kappa_i|\xi_i) - \E(
\kappa_i).
\]
By Lemma~\ref{lem:squareFunEst2},
%
\begin{eqnarray}
\label{eqn:1pq2EstPf} &&\biggl(\E \biggl\|\sum_i \E(
\eta_{i}|\xi_i) - \E(\eta_{i}) \biggr\|
_{L^q(S)}^p \biggr)^{{1}/{p}}\nonumber \\
&&\qquad \lesssim_{p,q} \biggl\|
\biggl(\sum_i \E\bigl|\E (\eta_{i}|
\xi_i) - \E(\eta_{i})\bigr|^2
\biggr)^{{1}/{2}}\biggr \|_{L^q(S)}
\\
&&\qquad \leq 2\biggl \| \biggl(\sum_i \E|
\eta_{i}|^2 \biggr)^{{1}/{2}}\biggr \|_{L^q(S)},\nonumber
\end{eqnarray}
where the final step follows from the triangle inequality and Jensen's
inequality. Now apply Lemma~\ref{lem:typeCotypeEst} [using that
$L^q(S)$ has type $q$], (\ref{eqn:littlesInf1}) and Jensen's inequality
to find
\begin{eqnarray*}
&&\biggl(\E \biggl\|\sum_i \E(\theta_i|
\xi_i) - \E(\theta_i) \biggr\| _{L^q(S)}^p
\biggr)^{{1}/{p}} \\
&&\qquad \lesssim_{p,q} \biggl(\Ex \biggl(\sum
_i \bigl\|\E(\theta_i|\xi_i) - \E (
\theta _i)\bigr\|_{L^q(S)}^q \biggr)^{{p}/{q}}
\biggr)^{{1}/{p}}
\\
&&\qquad \leq \biggl(\sum_i \E\bigl\|\E(
\theta_i|\xi_i) - \E(\theta_i)\bigr\|
_{L^q(S)}^p \biggr)^{{1}/{p}}
\\
&&\qquad \leq 2 \biggl(\sum_i \E\|\theta_i
\|_{L^q(S)}^p \biggr)^{{1}/{p}}.
\end{eqnarray*}
Similarly, Lemma~\ref{lem:typeCotypeEst}, (\ref{eqn:littlesInf0}) and
Jensen's inequality yield
\[
\biggl(\E \biggl\|\sum_i \E(\kappa_i|
\xi_i) - \E(\kappa_i)\biggr \| _{L^q(S)}^p
\biggr)^{{1}/{p}} \lesssim_{p,q} \biggl(\sum
_i \E\| \kappa_i\| _{L^q(S)}^q
\biggr)^{{1}/{q}}.
\]
The asserted estimate now follows by the triangle inequality.

\emph{Case $1\leq q\leq p\leq2$}: The proof is very similar to the
previous case. Let $(\eta_i) \in S_q$ and $(\theta_i) \in D_{p,q}\cap
D_{q,q}$ be such that $\xi_i = \eta_i + \theta_i$, then
\[
\xi_i = \E(\eta_{i}|\xi_i) - \E(
\eta_{i}) + \E(\theta_i|\xi_i) - \E (
\theta_i).
\]
By the same argument as in (\ref{eqn:1pq2EstPf}),
\[
\biggl(\E \biggl\|\sum_i \E(\eta_{i}|
\xi_i) - \E(\eta_{i}) \biggr\| _{L^q(S)}^p
\biggr)^{{1}/{p}} \lesssim_{p,q} \biggl\| \biggl(\sum
_i \E|\eta _{i}|^2
\biggr)^{{1}/{2}}\biggr \|_{L^q(S)}.
\]
Moreover, successively applying Lemma~\ref{lem:typeCotypeEst}, the
Rosenthal inequality (\ref{eqn:RosPos}) (using that $q\leq p$) and
Jensen's inequality yields
\begin{eqnarray*}
&& \biggl(\E\biggl \|\sum_i \E(\theta_i|
\xi_i) - \E(\theta_i) \biggr\| _{L^q(S)}^p
\biggr)^{{1}/{p}}
\\
&&\qquad \lesssim_{p,q} \biggl(\Ex \biggl(\sum_i
\bigl\|\E(\theta_i|\xi_i) - \E (\theta _i)
\bigr\|_{L^q(S)}^q \biggr)^{{p}/{q}} \biggr)^{{1}/{p}}
\nonumber
\\
&&\qquad \lesssim_{p,q} \max \biggl\{ \biggl(\sum
_i \E\bigl\|\E(\theta_i|\xi_i) - \E (
\theta_i)\bigr\|_{L^q(S)}^p \biggr)^{{1}/{p}},
\\
&&\hspace*{41pt}\quad\qquad \biggl(\sum_i \E\bigl\|\E(\theta_i|
\xi_i) - \E(\theta_i)\bigr\| _{L^q(S)}^q
\biggr)^{{1}/{q}} \biggr\}
\\
&&\qquad \leq2 \max \biggl\{ \biggl(\sum_i \E\|
\theta_i\|_{L^q(S)}^p \biggr)^{{1}/{p}},
\biggl(\sum_i \E\|\theta_i
\|_{L^q(S)}^q \biggr)^{1/q} \biggr\}.
\end{eqnarray*}
The result now follows by the triangle inequality.

\emph{Case $1\leq q\leq2\leq p<\infty$}: By Hoffmann--J{\o}rgensen's
inequality (\ref{eqn:HJTKSIntro}), we have
\[
\biggl(\E\biggl \|\sum_i \xi_i
\biggr\|_{L^q(S)}^p \biggr)^{{1}/{p}} \lesssim _p
\max \biggl\{ \biggl(\E \biggl\|\sum_i
\xi_i \biggr\|_{L^q(S)}^q \biggr)^{1/q}, \Bigl(
\E\max_i \|\xi_i\|_{L^q(S)}^p
\Bigr)^{{1}/{p}} \biggr\}.
\]
By the previous case (with $p=q$), we have
\[
\biggl(\E \biggl\|\sum_i \xi_i
\biggr\|_{L^q(S)}^q \biggr)^{{1}/{q}} \simeq _{p,q} \bigl\|(
\xi_i)\bigr\|_{S_{q}+D_{q,q}}
\]
and obviously
\[
\Bigl(\E\max_i \|\xi_i\|_{L^q(S)}^p
\Bigr)^{{1}/{p}} \leq \biggl(\sum_i \E\|
\xi_i\|_{L^q(S)}^p \biggr)^{{1}/{p}}.
\]

\emph{Case $1\leq p\leq2\leq q<\infty$}: Let $\xi_i = \eta_i +
\theta
_i$ with $(\eta_i)\in S_q\cap D_{q,q}$ and $(\theta_i) \in D_{p,q}$.
Then, $\xi_i = \E(\eta_i|\xi_i) - \E(\eta_i) + \E(\theta_i|\xi
_i) - \E
(\theta_i)$. Since the $\E(\eta_i|\xi_i) - \E(\eta_i)$ are independent
and mean-zero, we can subsequently use H\"{o}lder's inequality and the
already established estimate in the case $p=q\geq2$ to find
\begin{eqnarray*}
&& \biggl(\E \biggl\|\sum_i \E(\eta_i|
\xi_i) - \E(\eta_i) \biggr\| _{L^q(S)}^p
\biggr)^{{1}/{p}}
\\
&&\qquad \leq \biggl(\E \biggl\|\sum_i \E(
\eta_i|\xi_i) - \E(\eta_i) \biggr\|
_{L^q(S)}^q \biggr)^{{1}/{q}}
\\
&&\qquad \lesssim_{p,q} \max \biggl\{ \biggl(\sum
_i \E\bigl\|\E(\eta_i|\xi_i) - \E(\eta
_i)\bigr\|_{L^q(S)}^q \biggr)^{{1}/{q}},
\\
&&\hspace*{40pt}\qquad\quad \biggl \| \biggl(\sum_i \E\bigl|\E(\eta_i|
\xi_i) - \E(\eta_i)\bigr|^2 \biggr)^{1/2}
\biggr\|_{L^q(S)} \biggr\}
\\
&&\qquad \leq2 \max \biggl\{ \biggl(\sum_i \E\|
\eta_i\|_{L^q(S)}^q \biggr)^{1/q}, \biggl\|
\biggl(\sum_i \E|\eta_i|^2
\biggr)^{{1}/{2}}\biggr \| _{L^q(S)} \biggr\}.
\end{eqnarray*}
On the other hand, as $L^q(S)$ has type $2$, it has type $p$ and
therefore Lemma~\ref{lem:typeCotypeEst} implies
\begin{eqnarray*}
&&\biggl(\E \biggl\|\sum_i \E(\theta_i|
\xi_i) - \E(\theta_i)\biggr \| _{L^q(S)}^p
\biggr)^{{1}/{p}} \\
&&\qquad \lesssim_{p,q} \biggl(\sum
_i \E\bigl\| \E (\theta_i|\xi_i) - \E(
\theta_i)\bigr\|_{L^q(S)}^p \biggr)^{{1}/{p}}
\\
&&\qquad \leq 2 \biggl(\sum_i \E\|\theta_i
\|_{L^q(S)}^p \biggr)^{{1}/{p}}.
\end{eqnarray*}
The claimed inequality now follows by the triangle inequality. This
completes the proof.
\end{pf*}

\section{It\^{o}-isomorphisms: Classical $L^q$-spaces}
\label{sec:ItoPoissonClas}

In this section, we present a proof of the It\^{o} isomorphism stated
in Theorem~\ref{thm1.1}. Let us first define the
Poisson stochastic integral.
%
\begin{definition}
Let $(\Om,\cF,\bP)$ be a probability space and let $(E,\cE,\mu)$
be a
measure space. We say that a random measure $N$ on $E$ is a \emph
{Poisson random measure} if the following conditions hold:
\begin{longlist}[(ii)]
\item[(i)] For disjoint $A_1,\ldots,A_n \in\cE$ the random variables
$N(A_1),\ldots,N(A_n)$ are independent and
\[
N \Biggl(\bigcup_{i=1}^n
A_i \Biggr) = \sum_{i=1}^n
N(A_i),
\]
\item[(ii)] For any $A \in\cE$ with $\mu(A)<\infty$ the random variable
$N(A)$ is Poisson distributed with parameter $\mu(A)$.
\end{longlist}
Let $\cE_{\mu} = \{A \in\cE \dvtx  \mu(A)<\infty\}$. Then the random
measure $\tilde{N}$ on $(E,\cE_{\mu},\mu)$ defined by
\[
\tilde{N}(A):= N(A) - \mu(A)\qquad (A \in\cE_{\mu}),
\]
is called the \emph{compensated Poisson random measure} associated
with $N$.
\end{definition}
As is well known, one can always construct a Poisson random measure on
any given $\sigma$-finite measure space $(E,\cE,\mu)$; see, for example,
\cite
{Sat99}.

Throughout, we let $(J,\cJ,\nu)$ be a $\sigma$-finite measure space
and we
fix a Poisson random measure $N$ on $\R_+\ti J$. To arrive at a
satisfactory stochastic integration theory with respect to the
associated compensated Poisson random measure, we need to impose the
following standard compatibility assumption.
%
\begin{assumption}
\label{ass:PoissonFiltration} Throughout we fix a filtration $(\cF
_t)_{t>0}$ such that for any $0\leq s<t<\infty$ and any $A \in\cJ$ the
random variable
$\tilde{N}((s,t]\ti A)$ is $\cF_t$-measurable and independent of $\cF_s$.
\end{assumption}
%
\begin{definition}
\label{def:simplePoissonSI} Fix a Banach space $X$ and let $F\dvtx \Om\ti
\R
_+\ti J\rightarrow X$. We say that $F$ is a \emph{simple, adapted
$X$-valued process} if there is a
finite partition $\pi=\{0=t_1<\cdots<t_{l+1}<\infty\}$ of $\R_+$,
$F_{i,j,k} \in L^{\infty}(\cF_{t_i})$, $x_{i,j,k} \in X$ and disjoint
sets $A_1,\ldots A_m$ in
$\cJ$ satisfying $\nu(A_j)<\infty$ for $i=1,\ldots,l$, $j=1,\ldots,m$
and $k=1,\ldots,n$ such that
%
\begin{equation}
\label{eqn:simpleBanach} F = \sum_{i=1}^l\sum
_{j=1}^m\sum_{k=1}^{n}
F_{i,j,k} \oto\chi_{(t_i,t_{i+1}]}\oto\chi_{A_j} \oto
x_{i,j,k}.
\end{equation}
Given $t>0$ and $B \in\cJ$, we define the \emph{\textup{(}compensated\textup{)} Poisson
stochastic integral of $F$} on $(0,t]\ti B$ with respect to $\tilde
{N}$ by
\[
\int_{(0,t]\ti B} F \,d\tilde{N} = \sum_{i=1}^l
\sum_{j=1}^m\sum
_{k=1}^n F_{i,j,k} \tilde{N}\bigl(
(t_{i}\wedge t,t_{i+1}\wedge t]\times(A_j
\cap B)\bigr) \oto x_{i,j,k},
\]
where $s\wedge t:=\min\{s,t\}$.
\end{definition}
The following elementary observation will be important for our proof.
The upper estimate in (\ref{eqn:PoissonMoments}) in the case $1\leq
p\leq2$ was noted earlier in \cite{BrH09}, Lemma~C.3.
%
\begin{lemma}
\label{lem:PoissonMoments}
Let $N$ be a Poisson distributed random variable with parameter $0\leq
\lambda\leq1$. Then for every $1\leq p<\infty$ there exist constants
$b_p, c_p>0$ such that
%
\begin{equation}
\label{eqn:PoissonMoments} b_p \lambda\leq\E|N-\lambda|^p \leq
c_p \lambda.
\end{equation}
\end{lemma}
\begin{pf}
The inequalities are trivial if $\lambda=0$, so we may assume $\lambda
>0$. Suppose first that $2\leq p<\infty$. We begin by proving the
inequality on the left-hand side of (\ref{eqn:PoissonMoments}). We have
%
\begin{eqnarray}
\label{eqn:momGen1} \E|N-\lambda|^p & =& \sum
_{k=0}^{\infty} |k-\lambda|^p
\frac
{\lambda^k
e^{-\lambda}}{k!}
\nonumber
\\[-8pt]
\\[-8pt]
\nonumber
& \geq&\sum_{k=2}^{\infty} |k-
\lambda|^2 \frac{\lambda^k
e^{-\lambda
}}{k!} + |\lambda|^p
e^{-\lambda} + |1-\lambda|^p\lambda e^{-\lambda}.
\end{eqnarray}
Hence,
%
\begin{eqnarray}
\label{eqn:momGen2}&& \E|N-\lambda|^p\nonumber\\
&&\qquad  \geq\E|N-\lambda|^2 - |
\lambda|^2 e^{-\lambda} - |1-\lambda|^2\lambda
e^{-\lambda} + |\lambda|^p e^{-\lambda} + |1-
\lambda|^p\lambda e^{-\lambda}
\nonumber
\\[-8pt]
\\[-8pt]
\nonumber
&&\qquad = \lambda+ \lambda e^{-\lambda}\bigl(-\lambda- (1-\lambda)^2 +
\lambda ^{p-1} + (1-\lambda)^p\bigr)
\\
&&\qquad = \lambda\bigl(1 + e^{-\lambda}f_p(\lambda)\bigr),\nonumber
\end{eqnarray}
where
%
\begin{equation}
\label{eqn:fp} f_p(\lambda) = \lambda^{p-1} -
\lambda^2 + \lambda- 1 + (1-\lambda)^p.
\end{equation}
One easily sees that $\min_{0\leq\lambda\leq1}(1+e^{-\lambda
}f_p(\lambda)) = b_p>0$. Indeed,
\[
1+e^{-\lambda}f_p(\lambda) > 1 + e^{-\lambda}\bigl(-
\lambda^2+\lambda -1\bigr) + e^{-\lambda}(1-\lambda)^p.
\]
Now,
\[
1 + e^{-\lambda}\bigl(-\lambda^2 + \lambda- 1\bigr) +
e^{-\lambda}(1-\lambda)^p > 0
\]
if and only if
\[
(1-\lambda)^p > -e^{\lambda} + \lambda^2 - \lambda+
1 = -2\lambda+ \frac{\lambda^2}{2} - \frac{\lambda^3}{6} - \frac{\lambda^4}{24} -
\cdots.
\]
Clearly, this holds if $0\leq\lambda\leq1$. This proves the left-hand
side inequality of (\ref{eqn:PoissonMoments}) if $2\leq p<\infty$. We
now consider the right-hand side inequality. It suffices to prove this
in the case where $p$ is an even integer $n$. Since the moment
generating function of $N-\lambda$ is given by
\[
\E\bigl(e^{t(N-\lambda)}\bigr) = e^{\lambda(e^t-1-t)} = \exp \Biggl(\lambda \sum
_{n=2}^\infty\frac{t^n}{n!} \Biggr),
\]
it is easy to see that the $n$th moment of $N-\lambda$ can be written
as $\lambda p_n(\lambda)$ for some polynomial $p_n$ with positive
coefficients. Since $\max_{0\leq\lambda\leq1}|p_n(\lambda)|\leq c_n$
for some constant $c_n>0$, our proof for the case $2\leq p<\infty$ is
complete.

Suppose now that $1\leq p<2$. Then, by the Cauchy--Schwartz inequality,
\begin{eqnarray*}
\lambda& =& \E|N-\lambda|^2 = \E|N-\lambda|^{p/2}|N-\lambda
|^{2-{p}/{2}}
\\
& \leq&\bigl(\E|N-\lambda|^p\bigr)^{{1}/{2}}\bigl(\E|N-
\lambda|^{4-p}\bigr)^{{1}/{2}}.
\end{eqnarray*}
Since $4-p\geq2$, we find by the above that
\[
\lambda^2 \leq\E|N-\lambda|^p\E|N-\lambda|^{4-p}
\leq\E |N-\lambda|^p c_{4-p}\lambda.
\]
To prove the right-hand side inequality in (\ref{eqn:PoissonMoments}),
note that if $1\leq p<2$ the inequalities in (\ref{eqn:momGen1}) and
(\ref{eqn:momGen2}) reverse and, therefore,
\[
\E|N-\lambda|^p \leq\lambda\max_{0\leq\lambda\leq
1}
\bigl(1+e^{-\lambda
}f_p(\lambda)\bigr),
\]
where $f_p$ is the continuous function defined in (\ref{eqn:fp}).
\end{pf}
%
\begin{remark}
\label{rem:assumptionSimpleProcess} By refining the partition $\pi$ in
Definition~\ref{def:simplePoissonSI} if necessary, we can and will
always assume that
$(t_{i+1}-t_i)\nu(A_j)\leq1$ for all $i=1,\ldots,l$, $j=1,\ldots,m$.
This will allow us to apply Lemma~\ref{lem:PoissonMoments} to the
compensated Poisson random
variables $\tilde{N}((t_{i}\wedge t,t_{i+1}\wedge t]\times(A_j\cap B))$.
\end{remark}
Let us finally record the following easy observation for further reference.
%
\begin{lemma}
\label{lem:conditionSimple} Suppose that $(E,\cE,\mu)$ is a $\sigma
$-finite measure space and let $X$ be a Banach space. Let $A_1,\ldots
,A_n$ be disjoint sets in $\Si$
satisfying $\mu(A_i)<\infty$ and let $\cA$ be the $\sigma$-algebra
generated by $A_1,\ldots,A_n$. Then, for any $G \in L^1(E;X)$ supported
on $\bigcup_{i=1}^n A_i$,
\[
\E(G|\cA) = \sum_{i=1}^n
\chi_{A_i} y_i
\]
for certain $y_i \in X$.
\end{lemma}
\begin{pf}
Let $A_{n+1}$ be the complement of $\bigcup_{i=1}^n A_i$ in $E$. Since
$A_1,\ldots,A_{n+1}$ are disjoint, $\cA$ is actually a finite algebra
consisting of $A_1,\ldots,A_{n+1}$ and all their possible unions.
Moreover, for any $1\leq i\leq n$,
\[
\int_{A_i} G \,d\mu= \int_{A_i} \sum
_{\{1\leq j\leq n  \dvtx  \mu
(A_j)\neq0\}} \bigl(\mu(A_j)
\bigr)^{-1} \biggl(\int_{A_j} G \,d\mu \biggr) \chi
_{A_j} \,d\mu.
\]
Since $\int_{A_{n+1}} G  \,d\mu=0$ by assumption, we conclude that
\[
\E(G|\cA) = \sum_{\{1\leq j\leq n  \dvtx  \mu(A_j)\neq0\}} \bigl(\mu (A_j)
\bigr)^{-1} \chi_{A_j} \int_{A_j} G \,d\mu.
\]
\upqed\end{pf}
We are now ready to prove Theorem~\ref{thm1.1}.
\begin{pf*}{Proof of Theorem~\ref{thm1.1}}
Using Assumption~\ref{ass:PoissonFiltration}, it is not difficult to
show that the process $(\int_{(0,s]\ti B} F  \,d\tilde{N})_{s>0}$ is a
martingale. Therefore, the map
\[
s\mapsto \biggl\|\int_{(0,s]\ti B} F \,d\tilde{N} \biggr\|_{L^q(S)}
\]
defines a positive submartingale in $L^p(\Om)$ and by Doob's maximal
inequality (see, e.g., \cite{ReY91}, Theorem 1.7) we have for any $p>1$,
\[
\biggl(\E\sup_{0<s\leq t} \biggl\|\int_{(0,s]\ti B} F \,d
\tilde {N} \biggr\| _{L^q(S)}^p \biggr)^{{1}/{p}} \leq
p' \biggl(\E \biggl\|\int_{(0,t]\ti B} F \,d\tilde{N}
\biggr\|_{L^q(S)}^p \biggr)^{{1}/{p}},
\]
where $\frac{1}{p} + \frac{1}{p'} = 1$. Moreover, $L^q(S)$ has the UMD
property if $1<q<\infty$, so in view of the decoupling inequalities
(\ref{eqn:decouplingSIIntro}) it suffices to prove
%
\begin{equation}
\label{eqn:PoissonSINoSup} \biggl(\E\E_c\biggl \|\int_{(0,t]\ti B} F \,d
\tilde{N}^c \biggr\| _{L^q(S)}^p \biggr)^{{1}/{p}}
\simeq_{p,q} \|F\chi_{(0,t]\ti B}\|_{\cI_{p,q}},
\end{equation}
where $\tilde{N}^c$ is an independent copy of $\tilde{N}$ on a
probability space $(\Om_c,\cF_c,\bP_c)$. We show this in the cases
$2\leq q\leq p<\infty$ and $1<p\leq q\leq2$ in detail. All the main
technical difficulties occur in these two cases. For the similar proof
in the other cases, we refer the reader to Appendix~\ref
{app:remainingCases}. Let $F$ be the simple adapted process given in
(\ref{eqn:simpleBanach}), taking Remark~\ref
{rem:assumptionSimpleProcess} into account. We may assume that
$t=t_{l+1}$ and $B=\bigcup_{j=1}^m A_j$. We write $\tilde{N}_{i,j}^c:=
\tilde{N}^c((t_i,t_{i+1}]\ti A_j)$ for brevity.

\emph{Case $2\leq q\leq p<\infty$}: Set $y_{i,j} = \sum_{k=1}^n
F_{i,j,k} \oto x_{i,j,k}$, then the doubly indexed sequence $d_{i,j} =
y_{i,j}\tilde{N}_{i,j}^c$ satisfies
\[
\int_{(0,t]\ti B} F \,d\tilde{N} = \sum_{i,j}
d_{i,j}.
\]
Moreover, for any fixed $\omega\in\Om$ the sequence $(d_{i,j}(\omega
))_{i,j}$ consists of independent, mean-zero random variables. By
applying Theorem~\ref{thm:summaryRosIntro} pointwise in $\Om$, we find
\begin{eqnarray*}
&&\biggl(\E_c \biggl\|\sum_{i,j}
d_{i,j} \biggr\|_{L^q(S)}^p \biggr)^{1/p}\\
&&\qquad
\simeq_{p,q} \max \biggl\{ \biggl\| \biggl(\sum_{i,j}
\E _{c}|d_{i,j}|^2 \biggr)^{{1}/{2}}
\biggr\|_{L^q(S)}, \biggl(\sum_{i,j} \E_{c}
\|d_{i,j}\|_{L^q(S)}^q \biggr)^{{1}/{q}}, \\
&&\hspace*{192pt}\biggl(
\sum_{i,j} \E_c\|d_{i,j}
\|_{L^q(S)}^p \biggr)^{{1}/{p}} \biggr\}
\end{eqnarray*}
and by taking the $L^p(\Om)$-norm on both sides we arrive at
\begin{eqnarray*}
&& \biggl(\E\E_c \biggl\|\sum_{i,j}
d_{i,j} \biggr\|_{L^q(S)}^p \biggr)^{1/p}
\\
&&\qquad \simeq_{p,q} \max \biggl\{ \biggl(\E\biggl \| \biggl(\sum
_{i,j} \E _{c}|d_{i,j}|^2
\biggr)^{{1}/{2}}\biggr \|_{L^q(S)}^p \biggr)^{{1}/{p}},
\\
&&\hspace*{39pt}\qquad\quad \biggl(\E \biggl(\sum_{i,j} \E_{c}
\|d_{i,j}\|_{L^q(S)}^q \biggr)^{p/q}
\biggr)^{{1}/{p}}, \biggl(\sum_{i,j} \E
\E_c\|d_{i,j}\| _{L^q(S)}^p
\biggr)^{{1}/{p}} \biggr\}.
\end{eqnarray*}
Using Lemma~\ref{lem:PoissonMoments} and Remark~\ref
{rem:assumptionSimpleProcess}, we compute
%
\begin{eqnarray}
\label{eqn:compSI1} && \biggl(\E \biggl\| \biggl(\sum_{i,j}
\E_{c}|d_{i,j}|^2 \biggr)^{1/2}
\biggr\|_{L^q(S)}^p \biggr)^{{1}/{p}}
\nonumber
\\
& &\qquad= \biggl(\E \biggl\| \biggl(\sum_{i,j}
|y_{i,j}|^2\E_c\bigl|\tilde {N}_{i,j}^c\bigr|^2
\biggr)^{{1}/{2}} \biggr\|_{L^q(S)}^p \biggr)^{{1}/{p}}
\\
&&\qquad = \biggl(\E \biggl\| \biggl(\sum_{i,j}
|y_{i,j}|^2(t_{i+1}-t_i) \nu
(A_j) \biggr)^{{1}/{2}} \biggr\|_{L^q(S)}^p
\biggr)^{{1}/{p}} = \|F\|_{\cS_q^p}\nonumber
\end{eqnarray}
and
%
\begin{eqnarray}
\label{eqn:compSI2} && \biggl(\E \biggl(\sum_{i,j}
\E_{c}\|d_{i,j}\|_{L^q(S)}^q
\biggr)^{p/q} \biggr)^{{1}/{p}}
\nonumber
\\
&&\qquad = \biggl(\E \biggl(\sum_{i,j} \|y_{i,j}
\|_{L^q(S)}^q\E_c\bigl|\tilde {N}_{i,j}^c\bigr|^q
\biggr)^{{p}/{q}} \biggr)^{{1}/{p}}
\\
&&\qquad \simeq_{q} \biggl(\E \biggl(\sum_{i,j}
\|y_{i,j}\|_{L^q(S)}^q(t_{i+1}-t_i)
\nu(A_j) \biggr)^{{p}/{q}} \biggr)^{{1}/{p}} = \|F
\|_{\cD_{q,q}^p}.\nonumber
\end{eqnarray}
Finally,
%
\begin{eqnarray}
\label{eqn:compSI3} &&\biggl(\sum_{i,j} \E\E_c
\|d_{i,j}\|_{L^q(S)}^p \biggr)^{{1}/{p}}\nonumber\\[2pt]
&&\qquad =
\biggl(\sum_{i,j} \E\|y_{i,j}
\|_{L^q(S)}^p\E_c\bigl|\tilde{N}_{i,j}^c\bigr|^p
\biggr)^{{1}/{p}}
\\[2pt]
&&\qquad \simeq_p \biggl(\sum_{i,j} \E
\|y_{i,j}\|_{L^q(S)}^p(t_{i+1}-t_i)
\nu (A_j) \biggr)^{{1}/{p}} = \|F\|_{\cD_{p,q}^p}.\nonumber
\end{eqnarray}
We conclude that (\ref{eqn:PoissonSINoSup}) holds.

\emph{Case $1<p\leq q\leq2$}: Let $\cI_{\mathrm{elem}}$ denote the
linear space of all simple functions on $\Om\ti\R_+\ti J\ti S$ with
support\vspace*{2pt} of finite measure. Note that $\cI_{\mathrm{elem}}$ is dense in
$\cS_{q}^p$, $\cD_{p,q}^p$ and $\cD_{q,q}^p$. Hence,\vspace*{1pt} if we fix $\eps
>0$, we can find a decomposition $F=F_1+F_2+F_3$ with $F_{\alpha} \in
\cI_{\mathrm{elem}}$ for $\al=1,2,3$ such that
\[
\|F_1\|_{\cS_q^p} + \|F_2\|_{\cD_{p,q}^p} +
\|F_3\|_{\cD_{q,q}^p} \leq \|F\|_{\cI_{p,q}} + \eps.
\]
Clearly, we may assume that $F_{1},F_2$ and $F_3$ have the same support
in $\R_+\ti J$ as~$F$. Let $\cA$ be the sub-$\sigma$-algebra of $\cB
(\R
_+)\ti\cJ$ generated by the sets $(t_i,t_{i+1}]\ti A_j$. The associated
conditional expectation $\E(\cdot|\cA)$ is well defined, as $\cJ$ is
$\sigma$-finite. By Lemma~\ref{lem:conditionSimple}, $\E(F_{\alpha
}|\cA)$
is of the form
\[
\E(F_{\alpha}|\cA) = \sum_{i,j,k}
F_{i,j,k,\alpha} \oto\chi _{(t_i,t_{i+1}]} \oto\chi_{A_j} \oto
x_{i,j,k,\alpha} \qquad(\alpha=1,2,3).
\]
Let $y_{i,j,\alpha} = \sum_{k=1}^n F_{i,j,k,\alpha} \oto
x_{i,j,k,\alpha
}$ and set $d_{i,j,\alpha} = y_{i,j,\alpha}N_{i,j}^c$, so that $d_{i,j}
= y_{i,j}N_{i,j}^c$ satisfies
\[
d_{i,j} = d_{i,j,1} + d_{i,j,2} + d_{i,j,3}.
\]
We apply Theorem~\ref{thm:summaryRosIntro} pointwise in $\Om$ to find
\begin{eqnarray*}
&& \biggl(\E_c \biggl\|\sum_{i,j}
d_{i,j} \biggr\|_{L^q(S)}^p \biggr)^{{1}/{p}}
\\[2pt]
&&\qquad \lesssim_{p,q} \biggl\| \biggl(\sum_{i,j}
\E_{c}|d_{i,j,1}|^2 \biggr)^{1/2}
\biggr\|_{L^q(S)}
\\[2pt]
&&\hspace*{13pt}\qquad\quad{} + \biggl(\sum_{i,j} \E_c
\|d_{i,j,2}\|_{L^q(S)}^p \biggr)^{1/p}
+
\biggl(\sum_{i,j} \E_{c}
\|d_{i,j,3}\|_{L^q(S)}^q \biggr)^{{1}/{q}}.
\end{eqnarray*}
By taking $L^p(\Om)$-norms on both sides and using the triangle
inequality, we obtain
\begin{eqnarray*}
&& \biggl(\E\E_c \biggl\|\sum_{i,j}
d_{i,j} \biggr\|_{L^q(S)}^p \biggr)^{1/p}
\\
&&\qquad \lesssim_{p,q} \biggl(\E \biggl\| \biggl(\sum_{i,j}
\E _{c}|d_{i,j,1}|^2 \biggr)^{{1}/{2}}
\biggr\|_{L^q(S)}^p \biggr)^{{1}/{p}}
\\
&&\hspace*{13pt}\qquad\quad{} + \biggl(\sum_{i,j} \E\E_c
\|d_{i,j,2}\|_{L^q(S)}^p \biggr)^{1/p} +
\biggl(\E \biggl(\sum_{i,j} \E_{c}
\|d_{i,j,3}\|_{L^q(S)}^q \biggr)^{p/q}
\biggr)^{{1}/{p}}.
\end{eqnarray*}
By the computations in (\ref{eqn:compSI1}), (\ref{eqn:compSI2}) and
(\ref{eqn:compSI3}),
%
\begin{eqnarray}
\label{eqn:normEstimatesFalpha} \biggl(\E \biggl\| \biggl(\sum_{i,j}
\E_{c}|d_{i,j,1}|^2 \biggr)^{1/2}
\biggr\|_{L^q(S)}^p \biggr)^{{1}/{p}} & =&\bigl \|\E(F_1|
\cA)\bigr\|_{\cS_q^p} \leq\| F_1\|_{\cS_q^p},
\nonumber
\\
\biggl(\sum_{i,j} \E\E_c
\|d_{i,j,2}\|_{L^q(S)}^p \biggr)^{{1}/{p}} &
\simeq_p &\bigl\|\E(F_2|\cA)\bigr\|_{\cD_{p,q}^p} \leq
\|F_2\|_{\cD_{p,q}^p},
\\
\biggl(\E \biggl(\sum_{i,j} \E_{c}
\|d_{i,j,3}\|_{L^q(S)}^q \biggr)^{p/q}
\biggr)^{{1}/{p}} & \simeq_q&\bigl\|\E(F_3|\cA)
\bigr\|_{\cD_{q,q}^p} \leq\|F_3\|_{\cD_{q,q}^p}.\nonumber
\end{eqnarray}
We conclude that
\begin{eqnarray*}
\biggl(\E\E_c \biggl\|\int_{(0,t]\ti B} F \,d
\tilde{N}^c \biggr\| _{L^q(S)}^p \biggr)^{{1}/{p}} &
\lesssim_{p,q}& \|F_1\|_{\cS_q^p} + \|F_2
\|_{\cD
_{p,q}^p} + \|F_3\|_{\cD_{q,q}^p}
\\
& \leq&\|F\|_{\cI_{p,q}} + \eps.
\end{eqnarray*}
We deduce the reverse inequality by duality. If $p',q'$ are the H\"
{o}lder conjugates of $p$ and $q$, then $(\cS_q^p)^*=\cS_{q'}^{p'}$,
$(\cD_{q,q}^p)^* = \cD_{q',q'}^{p'}$ and $(\cD_{p,q}^p)^* = \cD
_{p',q'}^{p'}$. Therefore, it follows from (\ref
{eqn:sumIntersectionDuality}) that $\cI_{p,q}^* = \cI_{p',q'}$. We let
\[
\langle F,G\rangle= \int_{\Om\ti\R_+\ti J\ti S} FG \,d\bP \,dt \,d\nu \,d\sigma
\]
denote the associated duality bracket. If $G \in\cI_{\mathrm{elem}}$
has the same support as $F$, then $\E(G|\cA)$ is of the form
\[
\E(G|\cA) = \sum_{i,j,k} G_{i,j,k}\oto
\chi_{(t_i,t_{i+1}]} \oto \chi _{A_j} \oto x_{i,j,k}^*,
\]
where $G_{i,j,k} \in L^{\infty}(\Om)$. Now,
%
\begin{eqnarray}
\label{eqn:simpleEstDualSI}
\langle F,G\rangle& =& \bigl\langle F,\E(G|\cA)\bigr\rangle
\nonumber
\\
& =& \sum_{i,j,k}\E(F_{i,j,k}
G_{i,j,k}) \,dt\ti \,d\nu\bigl((t_i,t_{i+1}]\ti
A_j\bigr) \bigl\langle x_{i,j,k},x_{i,j,k}^*\bigr\rangle
\\
& =& \sum_{i,j,k}\E(F_{i,j,k}
G_{i,j,k}) \,dt\ti \,d\nu\bigl((t_i,t_{i+1}]\ti
A_j\bigr) \bigl\langle x_{i,j,k},x_{i,j,k}^*\bigr\rangle
\nonumber
\\
& =& \sum_{i,j,k,l,m,n} \E(F_{i,j,k}
G_{l,m,n}) \E_c\bigl(\tilde {N}_{i,j}^c
\tilde{N}_{l,m}^c\bigr) \bigl\langle x_{i,j,k},x_{l,m,n}^*
\bigr\rangle
\nonumber
\\
& =& \sum_{i,j,k,l,m,n} \E\E_c
\bigl(F_{i,j,k} \tilde{N}_{i,j}^c G_{l,m,n}
\tilde{N}_{l,m}^c\bigl\langle x_{i,j,k},x_{l,m,n}^*
\bigr\rangle\bigr)
\nonumber
\\
& =& \biggl\langle\sum_{i,j,k} F_{i,j,k}
\tilde{N}_{i,j}^c x_{i,j,k}, \sum
_{l,m,n}G_{l,m,n} \tilde{N}_{l,m}^c
x_{l,m,n}^* \biggr\rangle
\nonumber
\\
& \leq& \biggl\|\int_{(0,t]\ti B} F \,d\tilde{N}^c
\biggr\|_{L^p(\Om\ti
\Om
_c;L^q(S))}
\nonumber
\\
&&{} \times \biggl\|\sum_{l,m,n}G_{l,m,n}
\tilde{N}_{l,m}^c x_{l,m,n}^*\biggr \| _{L^{p'}(\Om\ti\Om_c;L^{q'}(S))}.\nonumber
\end{eqnarray}
Since $2\leq q'\leq p'<\infty$, our previously established case
implies that
\[
\biggl\|\sum_{l,m,n}G_{l,m,n} \tilde{N}_{l,m}^c
x_{l,m,n}^* \biggr\| _{L^{p'}(\Om\ti\Om_c;L^{q'}(S))} \lesssim_{p,q}
 \bigl\|\E(G|\cA)\bigr\|
_{\cI
_{p',q'}} \leq\|G\|_{\cI_{p',q'}}.
\]
Summarizing, we find
\[
\langle F,G\rangle\lesssim_{p,q}\biggl \|\int_{(0,t]\ti B} F \,d
\tilde {N}^c \biggr\|_{L^p(\Om;L^q(S))}\|G\|_{\cI_{p',q'}}.
\]
Taking the supremum over all $G \in\cI_{\mathrm{elem}}$ yields the
result.

For the proof of the final assertion, note that $L^1(S)$ is not a UMD
space. However, for any $1\leq p<\infty$, the one-sided decoupling inequality
\[
\biggl(\E \biggl\|\int_{(0,t]\ti B} F \,d\tilde{N} \biggr\| _{L^1(S)}^p
\biggr)^{{1}/{p}} \lesssim_{p} \biggl(\E\E_c\biggl \|\int
_{(0,t]\ti B} F \,d\tilde{N}^c\biggr \|_{L^1(S)}^p
\biggr)^{{1}/{p}}
\]
still holds, see \cite{CoV12}. The remainder of the proof is the same
as in the case $q>1$.
\end{pf*}
%
\begin{remark}
It is clear from the proof that the inequality
\[
\E \biggl\|\int_{(0,t]\ti B} F \,d\tilde{N} \biggr\|_{L^q(S)} \lesssim
_{q} \| F\|_{\cI_{1,q}}
\]
is valid if $1\leq q<\infty$.
\end{remark}

\section{Preliminaries on noncommutative $L^q$-spaces}

We now turn to the extension of the It\^{o} isomorphism in Theorem~\ref
{thm1.1} to integrands taking values in a
noncommutative $L^q$-space. We begin by reviewing some facts on
noncommutative $L^q$-spaces. References for proofs of the results
presented below can be found in the survey \cite{PiX03}. Let $\cM$ be a
von Neumann algebra acting on a complex Hilbert space $H$, which is
equipped with a normal, semi-finite faithful trace $\tr$. We say that
a closed,
densely defined linear operator $x$ on $H$ is \emph{affiliated} with
the von Neumann algebra $\cM$ if $ux=xu$ for any unitary element $u$ in
the commutant $\cM'$ of~$\cM$. For such an operator, we define its \emph{distribution
function} by
\[
d(v;x) = \tr\bigl(e^{|x|}(v,\infty)\bigr)\qquad (v\geq0),
\]
where $e^{|x|}$ is the spectral measure of $|x|$. The \emph{decreasing
rearrangement} of $x$ is defined by
\[
\mu_t(x) = \inf\bigl\{v>0 \dvtx d(v;x)\leq t\bigr\}\qquad (t\geq0).
\]
We call $x$ \emph{$\tr$-measurable} if $d(v;x)<\infty$ for some $v>0$.
We let $S(\tr)$ denote the linear space of all $\tr$-measurable
operators. One can show that
$S(\tr)$ is a metrizable, complete topological $*$-algebra with respect
to the measure topology. Moreover, the trace $\tr$ extends to a trace
(again denoted by $\tr$) on
the set $S(\tr)_{+}$ of positive $\tr$-measurable operators by setting
%
\begin{equation}
\label{eqn:traceStau} \tr(x) = \int_0^{\infty}
\mu_t(x) \,dt\qquad \bigl(x \in S(\tr)_+\bigr).
\end{equation}
For $0<q<\infty$, we define
%
\begin{equation}
\label{eqn:NCLqNorm} \|x\|_{L^q(\cM)} = \bigl(\tr\bigl(|x|^q\bigr)
\bigr)^{{1}/{q}} \qquad\bigl(x \in S(\tr)\bigr).
\end{equation}
The linear space $L^{q}(\cM,\tr)$ of all $x \in S(\tr)$ satisfying
$\|
x\|_{L^q(\cM)}<\infty$ is called the \textit{noncommutative
$L^q$-space} associated with the pair
$(\cM,\tr)$. We usually denote $L^q(\cM,\tr)$ by $L^{q}(\cM)$ for
brevity. The map $\|\cdot\|_{L^q(\cM)}$ in (\ref{eqn:NCLqNorm}) defines
a norm (or $q$-norm if $0<q<1$)
on the space $L^q(\cM)$ under which it becomes a Banach space
(resp., quasi-Banach space). It can alternatively be viewed as
the completion of $\cM$ in the
(quasi-)norm $\|\cdot\|_{L^q(\cM)}$. We use the expression $L^{\infty
}(\cM)$ to denote $\cM$ equipped with its operator norm. By (\ref
{eqn:traceStau}) and using that
$\mu(|x|^q)=\mu(x)^q$, the noncommutative $L^q$-(quasi-)norm can
alternatively be computed as
\[
\|x\|_{L^q(\cM)} = \biggl(\int_0^{\infty}
\mu_t(x)^q \,dt \biggr)^{1/q} \qquad\bigl(x \in
L^q(\cM)\bigr).
\]
If $(S,\Si,\sigma)$ is a Maharam measure space, then $\cM=L^{\infty}(S)$
is a von Neumann algebra, which can be equipped with the normal,
semifinite faithful trace $\tr(f)=\int f  \,d\sigma$. In this case,
$L^q(\cM
)$ coincides with the usual Bochner space $L^q(S)$. Another familiar
example is obtained by taking $\cM=B(H)$, for a Hilbert space $H$. If
$B(H)$ is equipped with its standard trace, then the associated
noncommutative $L^q$-spaces are the usual Schatten spaces.

Below we shall use the following facts. First, recall H\"
{o}lder's inequality: if $0<q,r,s\leq\infty$ are such
that $\frac{1}{q} = \frac{1}{r} + \frac{1}{s}$ and $x \in L^r(\cM
)$, $y
\in L^s(\cM)$, then $xy \in L^q(\cM)$ and
%
\begin{equation}
\label{eqn:NCHolderLq} \|xy\|_{L^q(\cM)} \leq\|x\|_{L^r(\cM)}\|y\|
_{L^s(\cM)}.
\end{equation}
For $1\leq q<\infty$ and $\frac{1}{q} + \frac{1}{q'} = 1$, the familiar
duality $L^q(\cM)^*=L^{q'}(\cM)$ holds isometrically, with the duality
bracket given by $\langle x,y\rangle= \tr(xy)$. In particular,
$L^q(\cM
)$ is reflexive if and only if $1<q<\infty$ and $L^1(\cM)=\cM_*$
isometrically, where $\cM_*$ is the predual of $\cM$. We recall that
$L^q(\cM)$ is a UMD Banach space if and only if $1<q<\infty$. If
$1\leq
q<\infty$, then $L^q(\cM)$ has type $\min\{q,2\}$ and cotype $\max\{
q,2\}$.

We conclude this section by describing the column and row spaces and
their conditional versions. Let $1\leq q<\infty$. For
a finite sequence $(x_{i})$ in $L^{q}(\cM)$, we define
%
\begin{eqnarray}
\label{eqn:colRowNormLq} \bigl\|(x_{i})\bigr\|_{L^{q}(\cM;\ell^{2}_{c})} & =& \biggl\| \biggl(\sum
_{i} x_{i}^{*}x_{i}
\biggr)^{{1}/{2}}\biggr \|_{L^{q}(\cM)},
\nonumber
\\[-8pt]
\\[-8pt]
\nonumber
\bigl\|(x_{i})\bigr\|_{L^{q}(\cM;\ell^{2}_{r})} & =&\biggl \| \biggl(\sum
_{i} x_{i}x_{i}^{*}
\biggr)^{{1}/{2}} \biggr\|_{L^{q}(\cM)}.
\end{eqnarray}
Given $x_1,\ldots,x_n$, we let $\operatorname{diag}(x_i)$, $\operatorname
{row}(x_i)$ and $\operatorname{col}(x_i)$ denote the matrix with the $x_i$ on
its diagonal, first row and first column, respectively, and zeroes
elsewhere. Let $\cM\vNT B(\ell^2)$ be the von Neumann tensor product
equipped with its product trace $\tr\ot\operatorname{Tr}$. By noting that
\begin{eqnarray*}
\Biggl\| \Biggl(\sum_{i=1}^n
x_{i}^{*}x_{i} \Biggr)^{{1}/{2}}\Biggr \|
_{L^{q}(\cM)} & =
&\bigl\|\operatorname{col}(x_i)\bigr\|_{L^q(\cM\vNT B(\ell^2))},
\\
\Biggl\| \Biggl(\sum_{i=1}^n
x_{i}x_{i}^* \Biggr)^{{1}/{2}}\Biggr \| _{L^{q}(\cM
)} &
= &\bigl\|\operatorname{row}(x_i)\bigr\|_{L^q(\cM\vNT B(\ell^2))},
\end{eqnarray*}
one sees that the expressions in (\ref{eqn:colRowNormLq}) define two
norms on the linear space of all finitely nonzero sequences in $L^q(\cM
)$. The completions of this
space in these norms are called the \emph{column} and \emph{row
space}, 
respectively.

We shall need a conditional version of these two spaces. Suppose that
$\cN$ is a von Neumann algebra equipped with a normal, semifinite
faithful trace $\sigma$ and let $\cK$ be a von Neumann subalgebra such
that $\sigma|_{\cK}$ is again semifinite. Let $\cE\dvtx \cN\rightarrow
\cK$ be
the conditional expectation with respect to $\cK$. For a finite sequence
$(x_i)$ in $\cN$, we define
%
\begin{eqnarray}
\label{eqn:colRowNormLqConditional}
\bigl \|(x_i)\bigr\|_{L^q(\cN;\cE,\ell^2_c)} & =&\biggl \| \biggl(\sum
_i \cE |x_i|^2
\biggr)^{{1}/{2}} \biggr\|_{L^q(\cN)},
\nonumber
\\[-8pt]
\\[-8pt]
\nonumber
\bigl\|(x_i)\bigr\|_{L^q(\cN;\cE,\ell^2_r)} & =& \biggl\| \biggl(\sum
_i \cE \bigl|x_i^*\bigr|^2
\biggr)^{{1}/{2}} \biggr\|_{L^q(\cN)}.
\end{eqnarray}
Using techniques from Hilbert $C^*$-modules, it was shown by M. Junge
\cite{Jun02} that
\begin{eqnarray*}
&& \bigl\{(x_i)_{i=1}^n \dvtx x_i
\in\cN, n\geq1, \bigl\|(x_i)\bigr\|_{L^q(\cN
;\cE
,\ell^2_c)}<\infty\bigr\} \quad\mathrm{and}
\\
&& \bigl\{(x_i)_{i=1}^n \dvtx x_i
\in\cN, n\geq1, \bigl\|(x_i)\bigr\|_{L^q(\cN
;\cE
,\ell^2_r)}<\infty\bigr\}
\end{eqnarray*}
are normed linear spaces. By taking the completion of these spaces, we
obtain the \textit{conditional column} and \textit{row space},
respectively. Moreover, one can identify these spaces with complemented
subspaces of $L^q(\cK;\ell^2_c)$ and $L^q(\cK;\ell^2_r)$ and in this
way show that for any $1<q<\infty$ and $\frac{1}{q} + \frac{1}{q'}
= 1$
%
\begin{equation}
\label{eqn:dualcondColRow}\qquad \bigl(L^q\bigl(\cN;\cE,\ell^2_c
\bigr)\bigr)^* = L^{q'}\bigl(\cN;\cE,\ell^2_r
\bigr), \qquad\bigl(L^q\bigl(\cN;\cE,\ell ^2_r
\bigr)\bigr)^* = L^{q'}\bigl(\cN;\cE,\ell^2_c
\bigr),
\end{equation}
isometrically, with duality bracket given by
\[
\bigl\langle(x_i),(y_i)\bigr\rangle= \sum
_i \tr(x_i y_i).
\]
We refer to Section~2 of \cite{Jun02} for more information.

\section{$L^q$-valued Rosenthal inequalities: Noncommutative case}

In this section, we prove an extension of Theorem~\ref{thm:summaryRosIntro} for random variables taking values in a
noncommutative $L^q$-space. To state our main result, we introduce the
following norms on the linear space of all finite sequences $(f_i)$ of
random variables in $L^{\infty}(\Om;L^q(\cM))$, which serve as
substitutes for the norms considered in (\ref{eqn:normsRosCom}). First,
for $1\leq p,q<\infty$ we define
%
\begin{equation}
\label{eqn:DpqNormNC} \bigl\|(f_i)\bigr\|_{D_{p,q}} = \biggl(\sum
_i\E\|f_i\|_{L^q(\cM)}^p
\biggr)^{{1}/{p}}
\end{equation}
and we consider a column and row version of the space $S_q$ in
considered earlier, that is, we set
%
\begin{eqnarray}
\label{eqn:SqNormsNC} \bigl\|(f_i)\bigr\|_{S_{q,c}} & =& \biggl\| \biggl(\sum
_i \E|f_i|^2 \biggr)^{1/2}
\biggr\|_{L^q(\cM)},
\nonumber
\\[-8pt]
\\[-8pt]
\nonumber
\bigl\|(f_i)\bigr\|_{S_{q,r}} & =& \biggl\| \biggl(\sum
_i \E\bigl|f_i^*\bigr|^2
\biggr)^{1/2}\biggr \|_{L^q(\cM)}.
\end{eqnarray}
Here, $f_i^*$ denotes the (pointwise) adjoint of $f_i$. To see that the
latter two expressions define two norms, we identify them with a
particular instance of the conditional row and column norms in (\ref
{eqn:colRowNormLqConditional}). We let $\cN$ be the tensor product von
Neumann algebra $L^{\infty}(\Om)\vNT\cM$, equipped with the tensor
product trace $\E\ot\tr$. Let us recall that, for any $1\leq
q<\infty$,
the map defined on simple functions in the Bochner space $L^q(\Om
;L^q(\cM))$ by
\[
I_q \biggl(\sum_i
\chi_{A_i} x_i \biggr) = \sum_i
\chi_{A_i} \ot x_i
\]
extends to an isometric isomorphism
%
\begin{equation}
\label{eqn:tensorBochnerIden} L^q\bigl(\Om;L^q(\cM)\bigr) =
L^q\bigl(L^{\infty
}(\Om )\vNT\cM\bigr).
\end{equation}
Let $\cK$ be the von Neumann subalgebra of $\cN$ given by $\cK=\C
\id\ot
\cM$ and let $\cE$ be the associated conditional expectation. Under the
identification (\ref{eqn:tensorBochnerIden}), the element $\cE(f)$
coincides with the Bochner integral $\E(f)$, whenever $f \in L^q(\cN)$.
In particular, for any finite sequence $(f_i)$ in $\cN$,
\[
\label{eqn:abrevCondNorm} \bigl\|(f_i)\bigr\|_{L^q(\cN;\cE,\ell^2_c)} =
 \bigl\| (f_i)\bigr\|
_{S_{q,c}}, \qquad\bigl\|(f_i)\bigr\|_{L^q(\cN;\cE,\ell^2_r)} =\bigl \|(f_i)
\bigr\|_{S_{q,r}}.
\]
We denote by $D_{p,q}$, $S_{q,c}$ and $S_{q,r}$ the completion of the
linear space of all finite sequences $(f_i)$ of random variables in
$L^{\infty}(\Om;L^q(\cM))$ with respect to the norms in (\ref
{eqn:DpqNormNC}) and (\ref{eqn:SqNormsNC}). By (\ref
{eqn:dualcondColRow}), we have the duality
\[
(S_{q,c})^* = S_{q',r}, \qquad (S_{q,r})^* =
S_{q',c} \qquad\biggl(1<q<\infty, \frac
{1}{q} + \frac{1}{q'} = 1
\biggr).
\]
We are now ready to state the extension of Theorem~\ref{thm:summaryRosIntro}.
%
\begin{theorem}
\label{thm:summaryRosenthalLqNC} Let $1<p,q<\infty$. If $(\xi_i)$ is a
finite sequence of independent, mean-zero $L^q(\cM)$-valued random
variables, then
%
\begin{equation}
\label{eqn:summaryRosenthalLqNC} \biggl(\E \biggl\|\sum_i
\xi_i \biggr\|_{L^q(\cM)}^p \biggr)^{1/p} \simeq
_{p,q} \bigl\|(\xi_i)\bigr\|_{s_{p,q}},
\end{equation}
where $s_{p,q}$ is given by
\begin{eqnarray*}
S_{q,c} \cap S_{q,r} \cap D_{q,q} \cap
D_{p,q} &\qquad& \mathrm{if}\ 2\leq q\leq p<\infty,
\\
S_{q,c} \cap S_{q,r} \cap(D_{q,q} +
D_{p,q}) &\qquad& \mathrm{if}\ 2\leq p\leq q<\infty,
\\
(S_{q,c} \cap S_{q,r} \cap D_{q,q}) +
D_{p,q} &\qquad& \mathrm{if}\ 1<p<2\leq q<\infty,
\\
(S_{q,c} + S_{q,r} + D_{q,q}) \cap D_{p,q}&\qquad&
\mathrm{if}\ 1<q<2\leq p<\infty,
\\
S_{q,c} + S_{q,r} + (D_{q,q} \cap D_{p,q})&\qquad&
 \mathrm{if}\ 1<q\leq p\leq2,
\\
S_{q,c} + S_{q,r} + D_{q,q} + D_{p,q} &\qquad&
\mathrm{if}\ 1<p\leq q\leq2.
\end{eqnarray*}
\end{theorem}
To prove Theorem~\ref{thm:summaryRosenthalLqNC}, we shall need to
generalize Lemma~\ref{lem:squareFunEst2} and Theorem~\ref{thm:2pqqp}.
Let us first recall the noncommutative version of
Khintchine's inequalities (\ref{eqn:KhiComLp}).
%
\begin{theorem}[(Noncommutative Khintchine inequalities)]
\label{thm:KhintchineNC}  Let
$(r_i)$ be a Rademacher sequence and fix $1\leq p<\infty$. If $2\leq
q<\infty$, then, for any finite sequence $(x_i)$ in $L^q(\cM)$,
%
\begin{eqnarray}
\label{eqn:KhiNCUpperBigger2} && \biggl(\E \biggl\|\sum_i
r_i x_i\biggr \|_{L^q(\cM)}^p
\biggr)^{{1}/{p}}
\nonumber
\\[-8pt]
\\[-8pt]
\nonumber
& &\qquad\leq K_{p,q} \max \biggl\{ \biggl\| \biggl(\sum
_i|x_i|^2 \biggr)^{1/2}
\biggr\|_{L^q(\cM)}, \biggl\| \biggl(\sum_i\bigl|x_i^*\bigr|^2
\biggr)^{{1}/{2}} \biggr\| _{L^q(\cM)} \biggr\}
\end{eqnarray}
and
\begin{eqnarray*}
&& \biggl(\E \biggl\|\sum_i r_i
x_i\biggr \|_{L^q(\cM)}^2 \biggr)^{1/2} 
\\
&&\qquad \geq
\max \biggl\{\biggl \| \biggl(\sum_i|x_i|^2
\biggr)^{{1}/{2}}\biggr \| _{L^q(\cM)}, \biggl\| \biggl(\sum
_i\bigl|x_i^*\bigr|^2 \biggr)^{{1}/{2}}
\biggr\|_{L^q(\cM
)} \biggr\}.
\end{eqnarray*}
On the other hand, if $1\leq q\leq2$, then
\begin{eqnarray*}
&& \biggl(\E \biggl\|\sum_i r_i
x_i \biggr\|_{L^q(\cM)}^2 \biggr)^{1/2} 
\\
&&\qquad \leq
\inf \biggl\{ \biggl\| \biggl(\sum_i|y_i|^2
\biggr)^{{1}/{2}} \biggr\| _{L^q(\cM)} + \biggl\| \biggl(\sum
_i\bigl|z_i^*\bigr|^2 \biggr)^{{1}/{2}}
\biggr\|_{L^q(\cM
)} \biggr\}
\end{eqnarray*}
and
\begin{eqnarray*}
&& \biggl(\E \biggl\|\sum_i r_i
x_i \biggr\|_{L^q(\cM)}^q \biggr)^{1/q}
\\
&&\qquad \gtrsim_{p,q} \inf \biggl\{ \biggl\| \biggl(\sum
_i|y_i|^2 \biggr)^{1/2}
\biggr\|_{L^q(\cM)} +\biggl \| \biggl(\sum_i\bigl|z_i^*\bigr|^2
\biggr)^{{1}/{2}}\biggr \|_{L^q(\cM
)} \biggr\},
\end{eqnarray*}
where the infimum is taken over all decompositions $x_i = y_i + z_i$ in
$L^q(\cM)$.
\end{theorem}
%
\begin{remark}
Theorem~\ref{thm:KhintchineNC} was proved for $p=q$ in \cite
{Lus86,LuP91}. The general case immediately follows by applying
Kahane's inequalities (\ref{eqn:Kahane}). It is known that the constant
$\kappa_{p,q}$ in (\ref{eqn:Kahane}) satisfies $\kappa_{p,q}\leq
(p-1)^{1/2}/(q-1)^{1/2}$ if $1<q<p<\infty$ (see, e.g., \cite{PeG99},
Theorem 3.1). It was proved by Buchholz that $K_{2n}^{2n}=(2n)!/(2^n
n!)$ if $n \in\N$ (\cite{Buc01}, Theorem 5 and the remark following
it). From this, it follows that $K_{q,q}<\sqrt{q}$ if $q\geq2$.
Summarizing, if $2\leq q<p<\infty$, then
\[
K_{p,q}\leq\kappa_{p,q}K_{q,q}\leq(p-1)^{1/2}/(q-1)^{1/2}q^{1/2}
\leq\sqrt{2}\sqrt{p-1}
\]
and if $2\leq p\leq q<\infty$, then $K_{p,q}\leq K_{q,q}<\sqrt{q}$.
\end{remark}
In the proof of the next result, we use for $0<q\leq1$ and $\xi\in
L^1(\Om;L^q(\cM)_+)$,
%
\begin{equation}
\label{eqn:ExpLqLess1} \E\|\xi\|_{L^q(\cM)} \leq\|\E\xi\| _{L^q(\cM)}.
\end{equation}
This follows by approximation by step functions using the inequality
\[
\|x+y\|_{L^q(\cM)} \geq\|x\|_{L^q(\cM)} + \|y\|_{L^q(\cM)}\qquad
\bigl(x,y \in L^q(\cM)_+\bigr).
\]
%
\begin{lemma}
\label{lem:squareFunEst2NC} Let $(\xi_i)$ be a finite sequence of
independent, mean-zero $L^q(\cM)$-valued random variables. If $1\leq
p,q<2$, then
\begin{eqnarray*}
& &\biggl(\E \biggl\|\sum_i \xi_i
\biggr\|_{L^q(\cM)}^p \biggr)^{1/p}
\\
&&\qquad \leq4 \inf \biggl\{ \biggl\| \biggl(\sum_i \E|
\eta_i|^2 \biggr)^{1/2} \biggr\|_{L^q(\cM)} + \biggl\|
\biggl(\sum_i \E\bigl|\theta_i^*\bigr|^2
\biggr)^{{1}/{2}} \biggr\|_{L^q(\cM)} \biggr\},
\end{eqnarray*}
where the infimum is taken over all sequences $(\eta_i) \in S_{q,c}$
and $(\theta_i) \in S_{q,r}$ such that $\xi_i = \eta_i + \theta_i$.
On the
other hand, if $2\leq p,q<\infty$, then
\begin{eqnarray*}
&& 2 \biggl(\E \biggl\|\sum_i \xi_i
\biggr\|_{L^q(\cM)}^p \biggr)^{1/p}
\\
&&\qquad \geq\max \biggl\{ \biggl\| \biggl(\sum_i \E|
\xi_i|^2 \biggr)^{1/2} \biggr\| _{L^q(\cM)}, \biggl\|
\biggl(\sum_i \E\bigl|\xi_i^*\bigr|^2
\biggr)^{{1}/{2}} \biggr\|_{L^q(\cM)} \biggr\}.
\end{eqnarray*}
\end{lemma}
\begin{pf}
Suppose $1\leq p,q<2$. Let $(\alpha_i)$ be a finite sequence in
$S_{q,c}$ of independent, mean-zero $L^q(\cM)$-valued random
variables. By
symmetrization (\ref{eqn:symmetrization}) and Theorem~\ref{thm:KhintchineNC},
\begin{eqnarray*}
\biggl(\E \biggl\|\sum_i \alpha_i
\biggr\|_{L^q(\cM)}^p \biggr)^{1/p} & \leq&2 \biggl(\E
\E_r \biggl\|\sum_i r_i
\alpha_i \biggr\|_{L^q(\cM
)}^p \biggr)^{{1}/{p}}
\\
& \leq&2 \biggl(\E \biggl\| \biggl(\sum_i |
\alpha_i|^2 \biggr)^{1/2} \biggr\|
_{L^q(\cM)}^p \biggr)^{{1}/{p}}
\\
& =& 2 \biggl(\E \biggl\|\sum_i |\alpha_i|^2
\biggr\|_{L^{q/2}(\cM
)}^{p/2} \biggr)^{{1}/{p}}
\\
& \leq&2 \biggl(\E \biggl\|\sum_i |
\alpha_i|^2 \biggr\|_{L^{q/2}(\cM
)} \biggr)^{{1}/{2}}
\\
& \leq&2 \biggl\|\sum_i \E|\alpha_i|^2
\biggr\|_{L^{{q}/{2}}(\cM
)}^{1/2} = 2 \biggl\| \biggl(\sum
_i \E|\alpha_i|^2
\biggr)^{1/2}\biggr \| _{L^q(\cM)}.
\end{eqnarray*}
Note that in the final two inequalities we apply Jensen's inequality and (\ref{eqn:ExpLqLess1}), respectively, using that
$\frac
{p}{2},\frac{q}{2}<1$.
Applying this for $(\al_i^*)$ yields
\[
\biggl(\E \biggl\|\sum_i \alpha_i
\biggr\|_{L^q(\cM)}^p \biggr)^{1/p} \leq \biggl\| \biggl(\sum
_i \E\bigl|\alpha_i^*\bigr|^2
\biggr)^{{1}/{2}} \biggr\| _{L^q(\cM)}.
\]
Let $(\eta_i)$ and $(\theta_i)$ be finite sequences in $S_{q,c}$ and
$S_{q,r}$, respectively, such that $\xi_i = \eta_i + \theta_i$, then
$\xi_i =
\E(\eta_i|\xi_i) - \E(\eta_i) + \E(\theta_i|\xi_i) - \E(\theta_i)$.
Since $(\E(\eta_i|\xi_i) - \E(\eta_i))$ and $(\E(\theta_i|\xi
_i) - \E
(\theta_i))$ are sequences of
independent, mean-zero random variables, we obtain by the triangle
inequality and the above,
\begin{eqnarray*}
\biggl(\E \biggl\|\sum_i \xi_i
\biggr\|_{L^q(\cM)}^p \biggr)^{1/p} & \leq& 2 \biggl\| \biggl(\sum
_i \E\bigl|\E(\eta_i|\xi_i) -
\E(\eta_i)\bigr|^2 \biggr)^{1/2} \biggr\|_{L^q(\cM)}
\\
&&{} + 2 \biggl\| \biggl(\sum_i \E\bigl|\E\bigl(
\theta_i^*\bigr|\xi_i\bigr) - \E\bigl(\theta
_i^*\bigr)\bigr|^2 \biggr)^{{1}/{2}} \biggr\|_{L^q(\cM)}.
\end{eqnarray*}
Therefore, by the triangle inequality in $S_{q,c}$ and $S_{q,r}$ we find
\begin{eqnarray*}
&& \biggl(\E \biggl\|\sum_i \xi_i
\biggr\|_{L^q(\cM)}^p \biggr)^{1/p}
\\
&&\qquad \leq2 \biggl( \biggl\| \biggl(\sum_i \E\bigl|\E(
\eta_i|\xi_i)\bigr|^2 \biggr)^{1/2}
\biggr\|_{L^q(\cM)} + \biggl\| \biggl(\sum_i \E\bigl|\E(\eta
_i)\bigr|^2 \biggr)^{{1}/{2}} \biggr\|_{L^q(\cM)}
\\
&&\qquad\qquad{} + \biggl\| \biggl(\sum_i \E\bigl|\E\bigl(
\theta_i^*\bigr|\xi_i\bigr)\bigr|^2
\biggr)^{1/2} \biggr\|_{L^q(\cM)} + \biggl\| \biggl(\sum
_i \E\bigl|\E\bigl(\theta_i^*\bigr)\bigr|^2
\biggr)^{1/2} \biggr\|_{L^q(\cM)} \biggr)
\\
&&\qquad \leq4 \biggl( \biggl\| \biggl(\sum_i \E|
\eta_i|^2 \biggr)^{1/2} \biggr\| _{L^q(\cM)} + \biggl\|
\biggl(\sum_i \E\bigl|\theta_i^*\bigr|^2
\biggr)^{1/2}\biggr \|_{L^q(\cM)} \biggr).
\end{eqnarray*}
Note that the final step follows directly from Kadison's
inequality for (noncommutative) conditional expectations if $\eta
_i,\theta_i$ are, in addition,
in $L^{\infty}\vNT\cM$. For general $\eta_i$ and $\theta_i$ as above,
the asserted inequality then follows by a density argument. This proves
the first statement.

Suppose now that $2\leq p,q<\infty$. By symmetrization (\ref
{eqn:symmetrization}) and Theorem~\ref{thm:KhintchineNC},
\begin{eqnarray*}
&& 2 \biggl(\E \biggl\|\sum_i \xi_i
\biggr\|_{L^q(\cM)}^p \biggr)^{1/p}
\\
&&\qquad \geq \biggl(\E\E_r \biggl\|\sum_i
r_i \xi_i \biggr\|_{L^q(\cM)}^p
\biggr)^{1/p}
\\
& &\qquad\geq\max \biggl\{ \biggl(\E\biggl \| \biggl(\sum_i |
\xi_i|^2 \biggr)^{1/2} \biggr\|_{L^q(\cM)}^p
\biggr)^{{1}/{p}}, \biggl(\E \biggl\| \biggl(\sum_i
\bigl|\xi_i^*\bigr|^2 \biggr)^{{1}/{2}} \biggr\|_{L^q(\cM)}^p
\biggr)^{1/p} \biggr\}
\\
&&\qquad = \max \biggl\{ \biggl(\E \biggl\|\sum_i |
\xi_i|^2 \biggr\|_{L^{q/2}(\cM
)}^{p/2}
\biggr)^{{1}/{p}}, \biggl(\E \biggl\|\sum_i \bigl|\xi
_i^*\bigr|^2 \biggr\|_{L^{{q}/{2}}(\cM)}^{p/2}
\biggr)^{{1}/{p}} \biggr\}
\\
&&\qquad \geq\max \biggl\{ \biggl\|\sum_i \E|
\xi_i|^2 \biggr\|_{L^{q/2}(\cM
)}^{{1}/{2}}, \biggl\|\sum
_i \E\bigl|\xi_i^*\bigr|^2
\biggr\|_{L^{q/2}(\cM
)}^{{1}/{2}} \biggr\}
\\
&&\qquad = \max \biggl\{ \biggl\| \biggl(\sum_i \E|
\xi_i|^2 \biggr)^{1/2} \biggr\| _{L^q(\cM)}, \biggl\|
\biggl(\sum_i \E\bigl|\xi_i^*\bigr|^2
\biggr)^{1/2} \biggr\| _{L^q(\cM)} \biggr\}.
\end{eqnarray*}
This completes the proof.
\end{pf}
For our discussion in Section~\ref{sec:randomMatrices}, we will keep
track of the dependence of the constants on $p$ and $q$ in the inequalities
(\ref{eqn:2pqqpNC}) and (\ref{eqn:2pqqpNCLower}) below.
%
\begin{theorem}
\label{thm:2pqqpNC} Suppose that $2\leq p,q<\infty$. If $(\xi_i)$ is a
finite sequence of independent, mean-zero $L^q(\cM)$-valued random
variables, then
%
\begin{eqnarray}
\label{eqn:2pqqpNC}&& \biggl(\E \biggl\|\sum_i
\xi_i \biggr\|_{L^q(\cM)}^p \biggr)^{1/p} \nonumber\\
&&\qquad \leq
C_{p,q}(1+\sqrt{2}) \max \biggl\{ \biggl\| \biggl(\sum
_i \E|\xi_i|^2
\biggr)^{1/2} \biggr\|_{L^q(\cM)},
\biggl\| \biggl(\sum_i \E\bigl|\xi_i^*\bigr|^2
\biggr)^{{1}/{2}} \biggr\| _{L^q(\cM
)},
\\
&&\hspace*{181pt}C_{{p}/{2},{q}/{2}} \biggl(\E \biggl(
\sum_i \|\xi_i\| _{L^q(\cM
)}^q
\biggr)^{{p}/{q}} \biggr)^{{1}/{p}} \biggr\},\hspace*{-5pt}\nonumber
\end{eqnarray}
where $C_{p,q} = 2 K_{p,q}<\max\{2\sqrt{2}\sqrt{p-1},2\sqrt{q}\}$ and
$K_{p,q}$ is the constant in (\ref{eqn:KhiNCUpperBigger2}). Moreover,
if $\kappa_{p,q}$ is the constant in (\ref{eqn:Kahane}) then
%
\begin{eqnarray}
\label{eqn:2pqqpNCLower} &&\biggl(\E \biggl\|\sum_i
\xi_i \biggr\|_{L^q(\cM)}^p \biggr)^{1/p} \nonumber
\\
&&\qquad \geq
\frac{1}{2} \max \biggl\{(\kappa_{q,p})^{-1} \biggl(\E
\biggl(\sum_i \| \xi_i\|
_{L^q(\cM)}^q \biggr)^{{p}/{q}} \biggr)^{{1}/{p}},
\\
&&\hspace*{44pt}\qquad
\biggl\| \biggl(\sum_i \E|\xi_i|^2
\biggr)^{{1}/{2}} \biggr\| _{L^q(\cM)},
\biggl\| \biggl(\sum
_i \E\bigl|\xi_i^*\bigr|^2
\biggr)^{{1}/{2}} \biggr\| _{L^q(\cM)} \biggr\}.\nonumber
\end{eqnarray}
\end{theorem}
\begin{pf}
We first prove (\ref{eqn:2pqqpNCLower}). By Lemma~\ref{lem:squareFunEst2NC},
\begin{eqnarray*}
&& \max \biggl\{ \biggl\| \biggl(\sum_i \E|
\xi_i|^2 \biggr)^{1/2} \biggr\| _{L^q(\cM)}, \biggl\|
\biggl(\sum_i \E\bigl|\xi_i^*\bigr|^2
\biggr)^{1/2} \biggr\| _{L^q(\cM)} \biggr\}
\\
&&\qquad \leq2 \biggl(\E \biggl\|\sum_i \xi_i
\biggr\|_{L^q(\cM)}^p \biggr)^{{1}/{p}}.
\end{eqnarray*}
By successively applying the cotype $q$ inequality for $L^q(\cM)$,
Kahane's inequalities (\ref{eqn:Kahane}) and (\ref
{eqn:symmetrization}), we see that
%
\begin{eqnarray}
\label{eqn:cotypeqConst1} &&\biggl(\E \biggl(\sum_i \|
\xi_i\|_{L^q(\cM)}^q \biggr)^{p/q}
\biggr)^{1/p} \nonumber\\
&&\qquad \leq \biggl(\E \biggl(\E_r \biggl\|\sum
_i r_i \xi_i
\biggr\|_{L^q(\cM)}^q \biggr)^{{p}/{q}} \biggr)^{{1}/{p}}
\nonumber\\[-8pt]\\[-8pt]\nonumber
&&\qquad \leq\kappa_{q,p} \biggl(\E\E_r \biggl\|\sum
_i r_i \xi_i \biggr\|
_{L^q(\cM
)}^p \biggr)^{{1}/{p}}
\nonumber\\
&&\qquad \leq2 \kappa_{q,p} \biggl(\E \biggl\|\sum_i
\xi_i \biggr\|_{L^q(\cM
)}^p \biggr)^{{1}/{p}}.
\nonumber
\end{eqnarray}
We refer to \cite{Fac87} for a proof that (\ref{eqn:cotypeqConst1})
holds with constant $1$.

We now prove (\ref{eqn:2pqqpNC}). By (\ref{eqn:symmetrization}) and
Theorem~\ref{thm:KhintchineNC}, we have
%
\begin{eqnarray}
\label{eqn:2pqqpNCFirst}&& \biggl(\E \biggl\|\sum_i
\xi_i \biggl\|_{L^q(\cM)}^p \biggr)^{1/p}
\nonumber
\\
&&\qquad \leq
2K_{p,q} \max \biggl\{ \biggl(\E \biggl\| \biggl(\sum
_i |\xi_i|^2 \biggr)^{1/2}
\biggr\|_{L^{q}(\cM)}^{p} \biggr)^{{1}/{p}},\\
&&\hspace*{87pt}
\biggl(\E \biggl\| \biggl(\sum_i \bigl|
\xi_i^*\bigr|^2 \biggr)^{1/2} \biggr\| _{L^{q}(\cM)}^{p}
\biggr)^{{1}/{p}} \biggr\}.\nonumber
\end{eqnarray}
By the triangle inequality in $L^{p/2}(\Om;L^{q/2}(\cM
))$, it follows that
%
\begin{eqnarray}
\label{eqn:2pqqpNCSquare1} \qquad&& \biggl(\E \biggl\| \biggl(\sum_i |
\xi_i|^2 \biggr)^{{1}/{2}} \biggr\| _{L^{q}(\cM
)}^{p}
\biggr)^{{1}/{p}}
\nonumber
\\
&&\qquad = \biggl(\E \biggl\|\sum_i |\xi_i|^2
\biggr\|_{L^{{q}/{2}}(\cM
)}^{p/2} \biggr)^{{1}/{p}}
\\
& &\qquad\leq \biggl( \biggl(\E\biggl \|\sum_i |
\xi_i|^2 - \E|\xi_i|^2 \biggr\|
_{L^{q/2}(\cM)}^{p/2} \biggr)^{2/p} + \biggl\|\sum
_i \E|\xi_i|^2 \biggr\|_{L^{{q}/{2}}(\cM)}
\biggr)^{{1}/{2}}.\nonumber
\end{eqnarray}
We now estimate the first term on the far right-hand side. By applying
(\ref{eqn:symmetrization}) and Theorem~\ref{thm:KhintchineNC} once
again, we obtain
%
\begin{eqnarray}
\label{eqn:2pqqpNCSquare2} && \biggl(\E \biggl\|\sum_i |
\xi_i|^2 - \E|\xi_i|^2
\biggr\|_{L^{q/2}(\cM
)}^{p/2} \biggr)^{2/p}
\nonumber
\\
&&\qquad \leq2K_{{p}/{2},{q}/{2}} \biggl(\E \biggl\| \biggl(\sum_i
\bigl| |\xi_i|^2 - \E|\xi_i|^2\bigr|^2
\biggr)^{{1}/{2}} \biggr\|_{L^{q/2}(\cM
)}^{p/2} \biggr)^{2/p}
\nonumber
\\[-9pt]
\\[-9pt]
\nonumber
& &\qquad\leq C_{{p}/{2},{q}/{2}} \biggl( \biggl(\E \biggl\| \biggl(\sum
_i |\xi _i|^4
\biggr)^{{1}/{2}} \biggr\|_{L^{{q}/{2}}(\cM)}^{p/2} \biggr)^{2/p}\\
&&\hspace*{82pt}{} +
\biggl\| \biggl(\sum_i \bigl|\E|\xi_i|^2\bigr|^2
\biggr)^{{1}/{2}} \biggr\|_{L^{{q}/{2}}(\cM
)} \biggr),\nonumber
\end{eqnarray}
where the final inequality is a consequence of the triangle inequality
in $L^{p/2}(\Om; L^{{q}/{2}}(\cM;\ell^2_c))$. Note that the
second term on the right-hand side is smaller than the first one. Indeed,
\begin{eqnarray}
\label{eqn:2pqqpNCSquare3}&& \biggl\| \biggl(\sum_i \bigl|\E|
\xi_i|^2\bigr|^2 \biggr)^{{1}/{2}} \biggr\|
_{L^{q/2}(\cM)} \nonumber\\
&&\qquad = \bigl\|\operatorname{col}\bigl(\E|\xi_i|^2
\bigr)\bigr\|_{L^{q/2}(\cM
\vNT B(\ell^2))}
\\
&&\qquad =\bigl \|\E\bigl(\operatorname{col}\bigl(|\xi_i|^2\bigr)\bigr)
\bigr\|_{L^{{q}/{2}}(\cM\vNT
B(\ell
^2))}
\nonumber
\\
&&\qquad \leq\E\bigl\|\operatorname{col}\bigl(|\xi_i|^2\bigr)
\bigr\|_{L^{{q}/{2}}(\cM\vNT
B(\ell
^2))}\nonumber
\\
&&\qquad \leq\bigl(\E\bigl\|\operatorname{col}\bigl(|\xi_i|^2\bigr)
\bigr\|_{L^{{q}/{2}}(\cM\vNT
B(\ell
^2))}^{p/2}\bigr)^{2/p}
\nonumber
\\
&&\qquad = \biggl(\E \biggl\| \biggl(\sum_i |
\xi_i|^4 \biggr)^{{1}/{2}} \biggr\| _{L^{{q}/{2}}(\cM)}^{p/2}
\biggr)^{2/p}.\nonumber
\end{eqnarray}
Write $x = \operatorname{col}(|\xi_i|)$ and $y=\operatorname{diag}(|\xi_i|)$ for
the matrices with the $|\xi_i|$ in their first column and diagonal,
respectively, and zeroes elsewhere. By the noncommutative H\"{o}lder
inequality (\ref{eqn:NCHolderLq}),
%
\begin{eqnarray}
\label{eqn:2pqqpNCSquare4} && \biggl(\E \biggl\| \biggl(\sum_i |
\xi_i|^4 \biggr)^{{1}/{2}} \biggr\| _{L^{q/2}(\cM)}^{p/2}
\biggr)^{2/p}
\nonumber
\\
& &\qquad= \bigl(\E\bigl\|\bigl(x^*y^*yx\bigr)^{{1}/{2}}
\bigr\|_{L^{{q}/{2}}(\cM\vNT B(\ell
^2))}^{p/2}
\bigr)^{2/p}
\nonumber
\\
&&\qquad = \bigl(\E\|yx\|_{L^{{q}/{2}}(\cM\vNT B(\ell^2))}^{p/2}\bigr)
^{2/p}
\nonumber
\\[-8pt]
\\[-8pt]
\nonumber
&&\qquad \leq\bigl(\E\bigl| \|y\|_{L^q(\cM\vNT B(\ell^2))}\|x\|_{L^q(\cM\vNT
B(\ell
^2))}\bigr|^{p/2}
\bigr)^{2/p}
\nonumber
\\
&&\qquad \leq\bigl(\E\|y\|_{L^q(\cM\vNT B(\ell^2))}^p\bigr)^{{1}/{p}}\bigl(\E
\|x\| _{L^q(\cM\vNT B(\ell^2))}^p\bigr)^{{1}/{p}}
\nonumber
\\
&&\qquad = \biggl(\E \biggl(\sum_i \|\xi_i
\|_{L^q(\cM)}^q \biggr)^{p/q} \biggr)^{{1}/{p}}
\biggl(\E \biggl\| \biggl(\sum_i |\xi_i|^2
\biggr)^{{1}/{2}} \biggr\|_{L^q(\cM)}^p \biggr)^{{1}/{p}}.\nonumber
\end{eqnarray}
Collecting our estimates (\ref{eqn:2pqqpNCSquare1}), (\ref
{eqn:2pqqpNCSquare2}), (\ref{eqn:2pqqpNCSquare3}) and (\ref
{eqn:2pqqpNCSquare4}), we find the quadratic
inequality
\[
a^2 \leq(2C_{{p}/{2},{q}/{2}}) ab + c^2,
\]
where we set $a = (\E\|(\sum_i |\xi_i|^2)^{{1}/{2}}\|_{L^q(\cM
)}^p)^{{1}/{p}}$, $b=(\E(\sum_i \|\xi_i\|_{L^q(\cM)}^q)^{p/q})^{{1}/{p}}$ and
$c=\|(\sum_i \E|\xi_i|^2)^{{1}/{2}}\|_{L^q(\cM)}$. Solving this
quadratic inequality, we obtain
\[
a\leq\tfrac{1}{2}\bigl(2C_{{p}/{2},{q}/{2}} b + \bigl((2C_{{p}/{2},{q}/{2}}b)^2
+ 4c^2\bigr)^{{1}/{2}}\bigr) \leq\tfrac{1+\sqrt {2}}{2} \max
\{2C_{{p}/{2},{q}/{2}} b, 2c\},
\]
that is,
\begin{eqnarray*}
&& \biggl(\E \biggl\| \biggl(\sum_i |
\xi_i|^2 \biggr)^{{1}/{2}} \biggr\| _{L^q(\cM
)}^p
\biggr)^{{1}/{p}}
\\
&&\qquad \leq(1+\sqrt{2}) \max \biggl\{ \biggl\| \biggl(\sum_i
\E|\xi _i|^2 \biggr)^{{1}/{2}}\biggr \|_{L^q(\cM)},\\
&&\hspace*{104pt}
C_{{p}/{2},{q}/{2}} \biggl(\E \biggl(\sum_i \|
\xi_i\|_{L^q(\cM)}^q \biggr)^{{p}/{q}}
\biggr)^{{1}/{p}} \biggr\}.
\end{eqnarray*}
Applying this to the sequence $(\xi_i^*)$, we obtain
\begin{eqnarray*}
&& \biggl(\E \biggl\| \biggl(\sum_i \bigl|
\xi_i^*\bigr|^2 \biggr)^{{1}/{2}} \biggr\| _{L^q(\cM
)}^p
\biggr)^{{1}/{p}}
\\
&&\qquad \leq(1+\sqrt{2}) \max \biggl\{ \biggl\| \biggl(\sum_i
\E\bigl|\xi _i^*\bigr|^2 \biggr)^{{1}/{2}}
\biggr\|_{L^q(\cM)},\\
&&\hspace*{103pt} C_{{p}/{2},{q}/{2}} \biggl(\E \biggl(\sum
_i \|\xi_i\|_{L^q(\cM)}^q
\biggr)^{{p}/{q}} \biggr)^{{1}/{p}} \biggr\}.
\end{eqnarray*}
Inequality (\ref{eqn:2pqqpNC}) now follows from (\ref{eqn:2pqqpNCFirst}).
\end{pf}
Note that even if $\cM$ is commutative, the proof of Theorem~\ref
{thm:2pqqpNC} is different from the one presented for Theorem~\ref{thm:2pqqp}.
We are now ready to prove Theorem~\ref{thm:summaryRosenthalLqNC}.
\begin{pf*}{Proof of Theorem~\ref{thm:summaryRosenthalLqNC}}
Observe that the spaces
$S_{q,c}$, $S_{q,r}$, $D_{p,q}$ and $D_{q,q}$ have dense intersection
and, therefore, the duality of these individual spaces imply together
with (\ref{eqn:sumIntersectionDuality}) that
\[
(s_{p,q})^* = s_{p',q'},\qquad \frac{1}{p} +
\frac{1}{p'} = 1,\qquad \frac
{1}{q}+\frac{1}{q'} = 1,
\]
with associated duality bracket
\[
\bigl\langle(f_i),(g_i)\bigr\rangle= \sum
_i \E\tr(f_i g_i).
\]
Thus, the lower estimates $\gtrsim_{p,q}$ in (\ref
{eqn:summaryRosenthalLqNC}) can be deduced from the upper ones using
the duality argument presented in (\ref{eqn:dualArgInd1}) and (\ref
{eqn:dualArgInd2}).

The upper estimates $\lesssim_{p,q}$ follow essentially as in the proof
of Theorem~\ref{thm:summaryRosIntro} once we replace the use of
Lemma~\ref{lem:squareFunEst2} and Theorem~\ref{thm:2pqqp} by their
noncommutative versions Lemma~\ref{lem:squareFunEst2NC} and
Theorem~\ref
{thm:2pqqpNC}, respectively. The straightforward modifications are left
to the reader.
\end{pf*}
Before deducing It\^{o} isomorphisms for Poisson stochastic integrals
taking values in a noncommutative $L^q$-space from Theorem~\ref
{thm:summaryRosenthalLqNC}, we take the opportunity to observe some
moment estimates for the norm of a sum of random matrices.

\section{Intermezzo on random matrices}
\label{sec:randomMatrices}

Let us recall the following noncommutative Khintchine inequality for
the operator norm of a Rademacher sum of matrices. Let $d_1,d_2 \in\N$
and set $d=\min\{d_1,d_2\}$. If $x_1,\ldots,x_n$ are $d_1\ti d_2$
random matrices, then there is a constant $C_{p,d}$ depending only on
$p$ and $d$ such that
%
\begin{equation}
\label{eqn:KIOperator}\Biggl(\E \Biggl\|\sum_{i=1}^n
r_i x_i\Biggr \| ^p \Biggr)^{{1}/{p}} \leq
C_{p,d}\max \Biggl\{ \Biggl\| \Biggl(\sum_{i=1}^n
|x_i|^2 \Biggr)^{{1}/{2}} \Biggr\|,\Biggl \| \Biggl(\sum
_{i=1}^n \bigl|x_i^*\bigr|^2
\Biggr)^{{1}/{2}} \Biggr\| \Biggr\}.\hspace*{-25pt}
\end{equation}
Indeed, this inequality can readily be deduced from the noncommutative
Khintchine inequalities for Schatten spaces. Since $\|x\|_{\log d} \leq
e\|x\|\leq e\|x\|_{\log d}$ for any $d_1\ti d_2$ matrix $x$,
\begin{eqnarray*}
&&\Biggl(\E \Biggl\|\sum_{i=1}^n
r_i x_i \Biggr\|^p \Biggr)^{{1}/{p}} \\
&&\qquad \leq
\Biggl(\E \Biggl\|\sum_{i=1}^n r_i
x_i \Biggr\|_{\log d}^{p} \Biggr)^{1/p}
\\
& & \qquad\leq K_{p,\log d}\max \Biggl\{ \Biggl\| \Biggl(\sum
_{i=1}^n |x_i|^2
\Biggr)^{1/2} \Biggr\|_{\log d},\Biggl \| \Biggl(\sum
_{i=1}^n \bigl|x_i^*\bigr|^2
\Biggr)^{{1}/{2}} \Biggr\|_{\log d} \Biggr\}
\\
&&\qquad \leq eK_{p,\log d}\max \Biggl\{ \Biggl\| \Biggl(\sum
_{i=1}^n |x_i|^2
\Biggr)^{{1}/{2}} \Biggr\|, \Biggl\| \Biggl(\sum_{i=1}^n
\bigl|x_i^*\bigr|^2 \Biggr)^{1/2} \Biggr\| \Biggr\}.
\end{eqnarray*}
By the remark following Theorem~\ref{thm:KhintchineNC}, if $2\leq\log
d \leq p$ then
\[
C_{p,d} \leq eK_{p,\log d} \leq e\sqrt{2}\sqrt{p-1}
\]
and $C_{p,d}\leq e\sqrt{\log d}$ if $2\leq p\leq\log d$.
%
\begin{remark}
\label{rem:Seginer}
The Khintchine inequality (\ref{eqn:KIOperator}) cannot hold with a
constant independent of the dimensions $d_1,d_2$. Indeed, it was shown
by Seginer (\cite{Seg00}, Theorem 3.1) that there is an absolute
constant $C$ such that for any $a_{ij}$, $i=1,\ldots,d_1$ $j=1,\ldots
,d_2$ in $\C$ and any $1\leq p\leq2\log\max\{d_1,d_2\}$ the rank one
matrices $x_{ij} = a_{ij}\ot e_{ij}$ satisfy
%
\begin{eqnarray}
\label{eqn:KISeg} && \biggl(\E \biggl\|\sum_{i,j}
r_{ij} x_{ij} \biggr\|^p \biggr)^{{1}/{p}}
\nonumber
\\[-10pt]
\\[-10pt]
\nonumber
&&\qquad \leq C(\log d)^{1/4}\max \biggl\{ \biggl\| \biggl(\sum
_{i,j} |x_{ij}|^2 \biggr)^{{1}/{2}}
\biggr\|, \biggl\| \biggl(\sum_{i,j} \bigl|x_{ij}^*\bigr|^2
\biggr)^{{1}/{2}}\biggr \| \biggr\}.
\end{eqnarray}
Moreover, the order of growth $(\log d)^{1/4}$ in (\ref
{eqn:KISeg}) is optimal (\cite{Seg00}, Theorem 3.2).
\end{remark}
%
\begin{theorem}
\label{thm:sumsRMs} Let $2\leq p<\infty$. If $(\xi_i)$ is a finite
sequence of independent, mean-zero $d_1\ti d_2$ random matrices, then
\begin{eqnarray*}
&&\biggl(\Ex\biggl \|\sum_i \xi_i
\biggr\|^p \biggr)^{{1}/{p}} \\
&&\qquad\leq 2(1+\sqrt {2})C_{p,d}
\max \biggl\{ \biggl\| \biggl(\sum_i \Ex|
\xi_i|^2 \biggr)^{1/2} \biggr\|, \biggl\| \biggl(\sum
_i \Ex\bigl|\xi_i^*\bigr|^2
\biggr)^{1/2}\biggr \|,\\
&&\hspace*{182pt}2C_{{p}/{2},d} \Bigl(\Ex\max_{i} \|\xi_i
\|^p \Bigr)^{1/p} \biggr\},
\end{eqnarray*}
where $d=\min\{d_1,d_2\}$. The reverse inequality holds with constant
$2^{1+1/p}$.
\end{theorem}
\begin{pf}
By repeating the proof of Theorem~\ref{thm:2pqqpNC} using (\ref
{eqn:KIOperator}) instead of the noncommutative Khintchine inequality
(\ref{eqn:KhiNCUpperBigger2}), we find
\begin{eqnarray*}
&&\biggl(\Ex\biggl \|\sum_i \xi_i
\biggr\|^p \biggr)^{{1}/{p}}  \\
&&\qquad\leq 2(1+\sqrt {2})C_{p,d}
\max \biggl\{ \biggl\| \biggl(\sum_i \Ex|
\xi_i|^2 \biggr)^{1/2} \biggr\|, \biggl\| \biggl(\sum
_i \Ex\bigl|\xi_i^*\bigr|^2
\biggr)^{1/2}\biggr \|,
\\
&&\hspace*{179pt}
2C_{{p}/{2},d} \bigl(\Ex\bigl\|\operatorname{diag}(\xi_i)
\bigr\|^p \bigr)^{1/p} \biggr\}.
\end{eqnarray*}
Clearly, $\|\operatorname{diag}(\xi_i)\| = \max_{i}\|\xi_i\|$, so the first
assertion holds.

For the second assertion, let $(r_i)$ be a Rademacher sequence on a
probability space $(\Om_r,\cF_r,\bP_r)$. Then
\begin{eqnarray*}
\Bigl(\E\max_{i} \|\xi_i\|_X^p
\Bigr)^{{1}/{p}} & =& \Bigl(\E\E _r\max_{i}
\|r_i \xi_i\|_X^p
\Bigr)^{{1}/{p}}
\\
& \leq&2^{{1}/{p}} \biggl(\E\E_r\biggl \|\sum
_i r_i\xi_i \biggr\|
_X^p \biggr)^{{1}/{p}}
\\
&\leq& 2^{1+{1}/{p}} \biggl(
\E \biggl\|\sum_i \xi_i
\biggr\|_X^p \biggr)^{{1}/{p}},
\end{eqnarray*}
where the first inequality follows by the L\'evy--Octaviani inequality
in \cite{KwW92}, Proposition 1.1.1. Moreover,
\begin{eqnarray*}
\biggl\| \biggl(\sum_i\E|\xi_i|^2
\biggr)^{{1}/{2}}\biggr \| & =& \biggl\|\E\E_r \sum
_{i,j} r_ir_j\xi_i^*
\xi_j \biggr\|^{{1}/{2}}
\\
& \leq& \biggl(\E\E_r \biggl\|\sum_{i,j}
r_ir_j\xi_i^*\xi_j \biggr\|
\biggr)^{1/2}
\\
& =& \biggl(\E\E_r \biggl\|\sum_{i}
r_i \xi_i \biggr\|^2 \biggr)^{1/2} \leq
2 \biggl(\E \biggl\|\sum_{i} \xi_i
\biggr\|^p \biggr)^{{1}/{p}},
\end{eqnarray*}
where the final inequality follows from (\ref{eqn:symmetrization}).
\end{pf}
As a consequence, we find the following moment inequalities for the
norm of a random matrix with independent, mean-zero entries.
%
\begin{corollary}
\label{cor:indepEntries} Let $2\leq p<\infty$. Suppose that $x_{ij}$,
$i=1,\ldots,d_1$, $j=1,\ldots,d_2$ are independent, mean-zero random
variables in $L^p(\Om)$. If
$x$ is the $d_1\times d_2$ random matrix $(x_{ij})$, then
%
\begin{eqnarray}
\label{eqn:indepEntriesUpper}&& \bigl(\Ex\|x\|^p\bigr)^{{1}/{p}}\nonumber\\
&&\qquad  \leq2(1+
\sqrt{2})C_{p,d} \max \Biggl\{ \max_{j=1,\ldots,d_2} \Biggl(\sum
_{i=1}^{d_1} \Ex x_{ij}^2
\Biggr)^{{1}/{2}}, \max_{i=1,\ldots,d_1} \Biggl(\sum
_{j=1}^{d_2} \Ex x_{ij}^2
\Biggr)^{1/2},
\\
&&\hspace*{217pt}2C_{{p}/{2},d} \Bigl(\Ex\max_{i,j} |x_{ij}|^p
\Bigr)^{1/p} \Biggr\},\nonumber
\end{eqnarray}
with $C_{p,d}<e\max\{\sqrt{\log d},\sqrt{2}\sqrt{p-1}\}$ as in
Theorem~\ref{thm:sumsRMs}.
\end{corollary}
\begin{pf}
Let $e_{ij}$ be the $d_1\ti d_2$ matrix having $1$ in entry $(i,j)$ and
zeroes elsewhere. Set $y_{ij} = x_{ij} \ot e_{ij}$, then $(y_{ij})$ is
a doubly indexed sequence of
independent, mean-zero random matrices and $x=\sum_{i,j} y_{ij}$.
Notice that
\[
y_{ij}^*y_{ij} = x_{ij}^2\ot
e_{ji}e_{ij} = x_{ij}^2\ot
e_{jj},
\]
so
\[
\biggl\| \biggl(\sum_{i,j} \E|y_{ij}|^2
\biggr)^{{1}/{2}}\biggr \| = \biggl\|\sum_{j} \biggl(\sum
_{i} \E x_{ij}^2
\biggr)^{{1}/{2}}\ot e_{jj} \biggr\| = \max_{j}
\biggl(\sum_{i} \Ex x_{ij}^2
\biggr)^{{1}/{2}}.
\]
Moreover,
\[
y_{ij}y_{ij}^* = x_{ij}^2\ot
e_{ij}e_{ji} = x_{ij}^2\ot
e_{ii}
\]
and, therefore,
\[
\biggl\| \biggl(\sum_{i,j} \E\bigl|y_{ij}^*\bigr|^2
\biggr)^{{1}/{2}} \biggr\| = \biggl\| \sum_{i} \biggl(
\sum_{j} \E x_{ij}^2
\biggr)^{{1}/{2}}\ot e_{ii} \biggr\| = \max_{i}
\biggl(\sum_{j} \Ex x_{ij}^2
\biggr)^{{1}/{2}}.
\]
Finally, it is clear that
\[
\Bigl(\Ex\max_{i,j}\|y_{ij}\|^p
\Bigr)^{{1}/{p}} = \Bigl(\Ex\max_{i,j}
|x_{ij}|^p \Bigr)^{{1}/{p}}.
\]
The result now follows from Theorem~\ref{thm:sumsRMs}.
\end{pf}
In \cite{Lat05}, Lata{\l}a showed that there is a universal constant
$C>0$ such that
%
\begin{eqnarray}
\label{eqn:Latala} \Ex\|x\| &\leq & C \Biggl(\max_{i=1,\ldots,d_1} \Biggl(\sum
_{j=1}^{d_2} \Ex x_{ij}^2
\Biggr)^{{1}/{2}} + \max_{j=1,\ldots,d_2} \Biggl(\sum
_{i=1}^{d_1} \Ex x_{ij}^2
\Biggr)^{{1}/{2}}
\nonumber
\\[-8pt]
\\[-8pt]
\nonumber
&&\hspace*{129pt}\qquad{}+ \biggl(\sum_{i,j} \Ex
x_{ij}^4 \biggr)^{1/4} \Biggr)
\end{eqnarray}
for any random matrix $x=(x_{ij})$ with independent, mean-zero entries
in $L^4(\Om)$. To compare this result to Corollary~\ref
{cor:indepEntries}, observe that (\ref{eqn:Latala}) implies together with~(\ref{eqn:HJTKSIntro}) that there is a universal constant $C>0$ such
that for all $1\leq p<\infty$,
%
\begin{eqnarray}
\label{eqn:Latala2} \bigl(\Ex\|x\|^p\bigr)^{{1}/{p}} & \leq& C
\frac{p}{\log p} \Biggl(\max_{i=1,\ldots,d_1} \Biggl(\sum
_{j=1}^{d_2} \Ex x_{ij}^2
\Biggr)^{{1}/{2}} + \max_{j=1,\ldots,d_2} \Biggl(\sum
_{i=1}^{d_1} \Ex x_{ij}^2
\Biggr)^{1/2}
\nonumber
\\[-8pt]
\\[-8pt]
\nonumber
&&{} + \biggl(\sum_{i,j} \Ex x_{ij}^4
\biggr)^{1/4} + \Bigl(\Ex \max_{i,j}
|x_{ij}|^p \Bigr)^{{1}/{p}} \Biggr).
\end{eqnarray}
The upper bound in Corollary~\ref{cor:indepEntries} exhibits different
growth behavior in $p$ and does not contain the factor $(\sum_{i,j}
\Ex
x_{ij}^4)^{1/4}$. In particular, the bound (\ref
{eqn:indepEntriesUpper}) is applicable to random matrices having
entries with infinite fourth moment. On the
other hand, note that the bound in (\ref{eqn:Latala2}) is of order
$\sqrt{d}$ for matrices with uniformly bounded entries, which is
optimal for $d\rightarrow\infty$ (see the discussion in \cite{Lat05}).
Through the use of the noncommutative Khintchine inequality in our
proof, we incur an extra factor of order $\sqrt{\log d}$. As the order
$(\log d)^{1/4}$ of the constant in (\ref{eqn:KISeg}) is
optimal, this additional factor is an inevitable product of our method.

\section{It\^{o}-isomorphisms: Noncommutative $L^q$-spaces}
\label{sec:ItoPoissonNC}

We now present an extension of Theorem~\ref{thm1.1} for
integrands taking values in a noncommutative $L^q$-space. In the
statement of our main result, we will use the following noncommutative
$L^2$-valued $L^q$-spaces, which were introduced by Pisier in \cite
{Pis98} and treated in more detail in \cite{JMX06}. For any simple
function on a measure space $(E,\cE,\mu)$ with values in $L^q(\cM)$,
$F=\sum_i \chi_{E_i} x_i$ say, we set
\begin{eqnarray*}
\|F\|_{L^q(\cM;L^2(\R_+\ti J)_c)} & =& \biggl\| \biggl(\sum_i
|x_i|^2 \mu (E_i) \biggr)^{{1}/{2}}
\biggr\|_{L^q(\cM)},
\\
\|F\|_{L^q(\cM;L^2(\R_+\ti J)_r)} & =& \biggl\| \biggl(\sum_i
\bigl|x_i^*\bigr|^2 \mu (E_i) \biggr)^{{1}/{2}}
\biggr\|_{L^q(\cM)}.
\end{eqnarray*}
It can be shown that these expression define two norms on the simple
functions, and we let $L^q(\cM;L^2(E)_c)$ and $L^q(\cM;L^2(E)_r)$
denote the respective completions in these norms. Alternatively, one
can describe these spaces as complemented subspaces of $L^q(\cM\vNT
B(L^2(E)))$ and in this way one can show that for $1<q,q'<\infty$ with
$\frac{1}{q} + \frac{1}{q'} = 1$,
%
\begin{eqnarray}
\label{eqn:DualRCL2} \bigl(L^q\bigl(\cM;L^2(E)_c
\bigr)\bigr)^* & =& L^{q'}\bigl(\cM;L^2(E)_r
\bigr),
\nonumber
\\[-8pt]
\\[-8pt]
\nonumber
\bigl(L^q\bigl(\cM;L^2(E)_r\bigr)\bigr)^* &
=& L^{q'}\bigl(\cM;L^2(E)_c\bigr).
\end{eqnarray}
We refer to Chapter~2 of \cite{JMX06} for details. Now, for any $1\leq
p,q<\infty$ we set
\[
\cS_{q,c}^p = L^p\bigl(\Om;L^q
\bigl(\cM;L^2(\R_+\ti J)_c\bigr)\bigr),\qquad \cS
_{q,r}^p = L^p\bigl(\Om;L^q\bigl(
\cM;L^2(\R_+\ti J)_r\bigr)\bigr).
\]
Since $L^q(\cM;L^2(\R_+\ti J)_c)$ and $L^q(\cM;L^2(\R_+\ti J)_r)$ can
be identified with closed subspaces of $L^q(\cM\vNT B(L^2(\R_+\ti
J)))$, they are reflexive if $1<q<\infty$. Therefore, it follows from
(\ref{eqn:DualRCL2}) that for any $1<p,q<\infty$,
%
\begin{equation}
\label{eqn:dualSqpNC} \qquad\bigl(\cS_{q,c}^p\bigr)^* =
\cS_{q',r}^{p'},\qquad \bigl(\cS_{q,r}^p
\bigr)^* = \cS_{q',c}^{p'},\qquad \biggl(\frac{1}{p} +
\frac{1}{p'} = 1, \frac{1}{q} + \frac{1}{q'} = 1\biggr).
\end{equation}
If $\cM$ is commutative, then $\cS_{q,c}^p$ and $\cS_{q,r}^p$ coincide
and are equal to the Bochner space $S_q^p=L^p(\Om;L^q(S;L^2(\R_+\ti
J)))$ considered earlier.

We are now ready to prove our main theorem.
%
\begin{theorem}
\label{thm:summarySILqPoissonNC} Let $1<p,q<\infty$. For any $B\in
\cJ$,
any $t>0$ and any simple, adapted $L^q(\cM)$-valued process $F$,
%
\begin{equation}
\label{eqn:summarySILqPoissonNC} \biggl(\E\sup_{0<s\leq t} \biggl\| \int
_{(0,s]\ti B} F \,d\tilde{N}\biggr \|_{L^q(\cM)}^p
\biggr)^{{1}/{p}} \simeq_{p,q} \|F\chi_{(0,t]\ti
B}
\|_{\cI_{p,q}},
\end{equation}
where $\cI_{p,q}$ is given by
\begin{eqnarray*}
\cS_{q,c}^p \cap\cS_{q,r}^p \cap
\cD_{q,q}^p \cap\cD_{p,q}^p &\qquad&
\mathrm{if}\ 2\leq q\leq p<\infty,
\\
\cS_{q,c}^p \cap\cS_{q,r}^p \cap
\bigl(\cD_{q,q}^p + \cD_{p,q}^p\bigr) &\qquad&
\mathrm{if}\ 2\leq p\leq q<\infty,
\\
\bigl(\cS_{q,c}^p \cap\cS_{q,r}^p
\cap\cD_{q,q}^p\bigr) + \cD_{p,q}^p &\qquad&
\mathrm{if}\ 1<p<2\leq q<\infty,
\\
\bigl(\cS_{q,c}^p + \cS_{q,r}^p +
\cD_{q,q}^p\bigr) \cap\cD_{p,q}^p &\qquad&
\mathrm {if}\ 1<q<2\leq p<\infty,
\\
\cS_{q,c}^p + \cS_{q,r}^p + \bigl(
\cD_{q,q}^p \cap\cD_{p,q}^p\bigr) &\qquad&
\mathrm {if}\ 1<q\leq p\leq2,
\\
\cS_{q,c}^p + \cS_{q,r}^p +
\cD_{q,q}^p + \cD_{p,q}^p &\qquad& \mathrm{if}\
1<p\leq q\leq2.
\end{eqnarray*}
\end{theorem}
\begin{pf}
The proof is similar to the one for Theorem~\ref
{thm1.1}, we sketch the main differences in the cases
$2\leq q\leq p<\infty$ and $1<p\leq q\leq2$. Since $L^q(\cM)$ is a UMD
space if $1<q<\infty$, by the decoupling inequality (\ref
{eqn:decouplingSIIntro}) and Doob's maximal inequality it suffices to
show that
\[
\biggl(\E\E_c \biggl\|\int_{(0,t]\ti B} F \,d
\tilde{N}^c \biggr\| _{L^q(\cM
)}^p \biggr)^{{1}/{p}}
\simeq_{p,q} \|F\chi_{(0,t]\ti B}\|_{\cI_{p,q}}.
\]
Let $F$ be the simple adapted process given in (\ref
{eqn:simpleBanach}), taking Remark~\ref{rem:assumptionSimpleProcess}
into account. We may assume that $t=t_{l+1}$ and $B=\bigcup_{j=1}^m A_j$.
We write $\tilde{N}_{i,j}^c:= \tilde{N}^c((t_i,t_{i+1}]\ti A_j)$ for
brevity.

\emph{Case $2\leq q\leq p<\infty$}: Set $y_{i,j} = \sum_{k=1}^n
F_{i,j,k} \oto x_{i,j,k}$ and $d_{i,j} = y_{i,j}\tilde{N}_{i,j}^c$,
then clearly
%
\begin{equation}
\label{eqn:stochIntIsSum} \int_{(0,t]\ti B} F \,d\tilde{N}^c = \sum
_{i,j} \,d_{i,j}.
\end{equation}
Moreover, for every fixed $\omega\in\Om$ the random variables
$d_{i,j}(\omega
)$ are independent and mean-zero. Therefore, we can apply Theorem~\ref
{thm:summaryRosenthalLqNC} pointwise in $\Om$ and subsequently take the
$L^p(\Om)$-norm on both sides to obtain
\begin{eqnarray*}
&& \biggl(\E\E_c \biggl\|\sum_{i,j}
d_{i,j} \biggr\|_{L^q(\cM)}^p \biggr)^{1/p}
\\
&&\qquad \lesssim_{p,q} \max \biggl\{ \biggl(\E \biggl\| \biggl(\sum
_{i,j} \E _{c}|d_{i,j}|^2
\biggr)^{{1}/{2}} \biggr\|_{L^q(\cM)}^p \biggr)^{1/p},\\
&&\hspace*{72pt}
\biggl(\E \biggl\| \biggl(\sum_{i,j} \E_{c}\bigl|d_{i,j}^*\bigr|^2
\biggr)^{1/2} \biggr\|_{L^q(\cM)}^p \biggr)^{{1}/{p}},
\\
&&\hspace*{72pt} \biggl(\E \biggl(\sum_{i,j} \E_{c}
\|d_{i,j}\|_{L^q(\cM)}^q \biggr)^{p/q}
\biggr)^{{1}/{p}}, \biggl(\sum_{i,j} \E
\E_c\|d_{i,j}\| _{L^q(\cM
)}^p
\biggr)^{{1}/{p}} \biggr\}
\\
&&\qquad \simeq_{p,q} \max\bigl\{\|F\|_{\cS_{q,c}^p}, \|F\|_{\cS_{q,r}^p}, \|F
\| _{\cD_{q,q}^p}, \|F\|_{\cD_{p,q}^p}\bigr\},
\end{eqnarray*}
where the final step follows by calculations analogous to (\ref
{eqn:compSI1}), (\ref{eqn:compSI2}) and (\ref{eqn:compSI3}).

\emph{Case $1<p\leq q\leq2$}: Let $\cI_{\mathrm{elem}}$ denote the
algebraic tensor product
\[
\cI_{\mathrm{elem}} = L^{\infty}(\Om) \ot L^{\infty}(\R_+)\ot
\bigl(L^1\cap L^{\infty}\bigr) (\cJ)\ot\bigl(L^1
\cap L^{\infty}\bigr) (\cM).
\]
Since this linear space is dense in $\cS_{q,c}^p$, $\cS_{q,r}^p$,
$\cD
_{p,q}^p$ and $\cD_{q,q}^p$, we can find, for any fixed $\eps>0$, a
decomposition $F=F_1+F_2+F_3+F_4$ with $F_{\alpha} \in\cI_{\mathrm
{elem}}$ such that
\[
\|F_1\|_{\cS_{q,c}^p} + \|F_2\|_{\cS_{q,r}^p} +
\|F_3\|_{\cD
_{p,q}^p} + \|F_4\|_{\cD_{q,q}^p} \leq\|F
\|_{\cI_{p,q}} + \eps.
\]
We may assume that the $F_{\al}$ have the same support in $\R_+\ti J$
as $F$. Let $\cA$ be the sub-$\sigma$-algebra of $\cB(\R_+)\ti\cJ$
generated by the sets $(t_i,t_{i+1}]\ti A_j$. By Lemma~\ref
{lem:conditionSimple} $\E(F_{\alpha}|\cA)$ is of the form
\[
\E(F_{\alpha}|\cA) = \sum_{i,j,k}
F_{i,j,k,\alpha} \oto\chi _{(t_i,t_{i+1}]} \oto\chi_{A_j} \oto
x_{i,j,k,\alpha}\qquad (\alpha=1,2,3,4).
\]
Let $y_{i,j,\alpha} = \sum_{k=1}^n F_{i,j,k,\alpha} \oto
x_{i,j,k,\alpha
}$ and set $d_{i,j,\alpha} = y_{i,j,\alpha}\tilde{N}_{i,j}^c$, then
(\ref{eqn:stochIntIsSum}) holds and
\[
d_{i,j} = d_{i,j,1} + d_{i,j,2} + d_{i,j,3} +
d_{i,j,4}.
\]
By computations similar to (\ref{eqn:compSI1}), (\ref{eqn:compSI2}) and
(\ref{eqn:compSI3}),
\begin{eqnarray*}
\bigl\|(d_{i,j,1})\bigr\|_{S_{q,c}^p} & =& \bigl\|\E(F_1|\cA)
\bigr\|_{\cS_{q,c}^p} \leq\| F_1\|_{\cS_{q,c}^p},
\\
\bigl\|(d_{i,j,2})\bigr\|_{S_{q,r}^p} & =& \bigl\|\E(F_2|\cA)
\bigr\|_{\cS_{q,r}^p} \leq\| F_2\|_{\cS_{q,r}^p},
\\
\bigl\|(d_{i,j,3})\bigr\|_{D_{p,q}^p} & \simeq&\hspace*{-2pt}_p\hspace*{2pt} \bigl\|
\E(F_3|\cA)\bigr\|_{\cD_{p,q}^p} \leq\|F_3\|_{\cD_{p,q}^p},
\\
\bigl\|(d_{i,j,4})\bigr\|_{D_{q,q}^p} & \simeq&\hspace*{-2pt}_q\hspace*{2pt} \bigl\|
\E(F_4|\cA)\bigr\|_{\cD_{q,q}^p} \leq\|F_4\|_{\cD_{q,q}^p}.
\end{eqnarray*}
By applying Theorem~\ref{thm:summaryRosenthalLqNC} pointwise in $\Om$
and subsequently taking $L^p(\Om)$-norms on both sides, we conclude that
\begin{eqnarray*}
&& \biggl(\E\E_c\biggl \|\int_{(0,t]\ti B} F \,d
\tilde{N}^c \biggr\| _{L^q(\cM
)}^p \biggr)^{{1}/{p}}
\\
&&\qquad \lesssim_{p,q} \|F_1\|_{\cS_{q,c}^p} +
\|F_2\|_{\cS_{q,r}^p} + \| F_3\| _{\cD_{p,q}^p} +
\|F_4\|_{\cD_{q,q}^p} \leq\|F\|_{\cI_{p,q}} + \eps.
\end{eqnarray*}
For the reverse estimate, observe that if $p',q'$ are the H\"{o}lder
conjugates of $p$ and~$q$, then in view of (\ref{eqn:dualSqpNC}) and
(\ref{eqn:sumIntersectionDuality}), we have $\cI_{p,q}^* = \cI
_{p',q'}$, with associated duality bracket
\[
\langle F,G\rangle= \int_{\Om\ti\R_+\ti J} \tr(FG) \,d\bP \,dt \,d\nu.
\]
The reverse inequality can therefore be deduced using the duality
argument (\ref{eqn:simpleEstDualSI}) explained in the proof of
Theorem~\ref{thm1.1}.
\end{pf}
Let us make a detailed comparison of our main result with the existing
results in the literature. We restrict our attention to \cite
{BrH09,Hau11,MPR10,MaR10} and refer to the references in these papers
for earlier achievements. In \cite{MPR10}, Marinelli, Pr{\'e}v{\^o}t and
R\"
{o}ckner showed using It\^{o}'s formula that if $H$ is a Hilbert space
and $2\leq p<\infty$, then
%
\begin{eqnarray}
\label{eqn:MPR} && \biggl(\E\sup_{0<s\leq t} \biggl\|\int
_{(0,s]\ti B} F \,d\tilde {N} \biggr\| _{H}^p
\biggr)^{{1}/{p}}
\nonumber\\[-8pt]  \\[-8pt]
&&\qquad \lesssim_{p,t} \biggl(\E\int_{(0,t]} \biggl(\int
_{B} \|F\|_H^2 \,d\nu
\biggr)^{p/2}\,dt \biggr)^{{1}/{p}}
+ \biggl(\E\int
_{(0,t]\ti B} \| F\| _H^p\, dt\,d\nu
\biggr)^{{1}/{p}}.\hspace*{-17pt}\nonumber
\end{eqnarray}
Due to the first term on the right-hand side, this estimate is only
near-optimal. Indeed, since
\[
\biggl(\E \biggl(\int_{(0,t]\ti B} \|F\|_H^2
\,d\nu \biggr)^{p/2}\,dt \biggr)^{{1}/{p}} \leq t^{{1}/{2}-{1}/{p}}
\biggl(\E\int_{(0,t]} \biggl(\int_{B} \|F
\|_H^2 \,d\nu \biggr)^{p/2}\,dt
\biggr)^{{1}/{p}},
\]
Theorem~\ref{thm:summarySILqPoissonNC} implies (\ref{eqn:MPR}) but not
vice versa. In \cite{MaR10}, Marinelli and R\"{o}ckner proved the bound
%
\begin{eqnarray}
\label{eqn:MaR} & &\biggl(\E\sup_{0<s\leq t} \biggl\|\int
_{(0,s]\ti B} F \,d\tilde {N} \biggr\| _{L^p(S)}^p
\biggr)^{{1}/{p}}
\nonumber
\\
& &\qquad\lesssim_{p,t} \biggl(\E\int_{(0,t]} \biggl(\int
_{B} \|F\| _{L^p(S)}^2 \,d\nu
\biggr)^{p/2}\,dt \biggr)^{{1}/{p}}\\
&&\hspace*{10pt}\qquad\quad{} + \biggl(\E\int
_{(0,t]\ti
B} \| F\|_{L^p(S)}^p\, dt\,d\nu
\biggr)^{{1}/{p}},\nonumber
\end{eqnarray}
valid for any $2\leq p<\infty$. This result is deduced by a Fubini-type
argument from the estimate (\ref{eqn:MPR}) for $H=\R$. Of course, such
an argument can only work if $p=q$ (in our notation). Observe that the
optimal bound in Theorem~\ref{thm:summarySILqPoissonNC} improves upon~(\ref{eqn:MaR}). Also note that the constants in (\ref
{eqn:summarySILqPoissonNC}) do not depend on $t$, in contrast to (\ref
{eqn:MPR}) and (\ref{eqn:MaR}). Finally, let us recall the following
bounds valid for a Banach space $X$ with martingale type $1<q\leq2$.
Brze\'{z}niak and Hausenblas showed (\cite{BrH09}, Corollary B.6) that
if $1<p\leq q$ then
%
\begin{eqnarray}
\label{eqn:BrH}&& \biggl(\E\sup_{0<s\leq t} \biggl\|\int
_{(0,s]\ti B} F \,d\tilde {N}\biggr \| _{X}^p
\biggr)^{{1}/{p}}
\lesssim_{p,q,X} \biggl(\E \biggl(\int
_{(0,t]\ti
B} \|F\|_X^q \,dt\,d\nu
\biggr)^{{p}/{q}} \biggr)^{{1}/{p}}.\hspace*{-24pt}
\end{eqnarray}
Moreover, Hausenblas proved (\cite{Hau11}, Proposition 2.14) that if
$p=q^n$ for some $n\in\N$, then
%
\begin{eqnarray}
\label{eqn:Hau} && \biggl(\E\sup_{0<s\leq t} \biggl\|\int
_{(0,s]\ti B} F \,d\tilde {N} \biggr\| _{X}^p
\biggr)^{{1}/{p}}
\nonumber\\[-4pt]\\[-12pt]\nonumber
&&\qquad\lesssim_{p,q,X} \biggl(\E \biggl(\int_{(0,t]\ti B} \|F
\|_X^{q} \,dt\,d\nu \biggr)^{{p}/{q}}
\biggr)^{{1}/{p}}
+ \biggl(\E\int_{(0,t]\ti B} \| F\|
_X^{p} \,dt\,d\nu \biggr)^{{1}/{p}}.\nonumber\hspace*{-8pt}
\end{eqnarray}
If $X=L^2(\cM)$, so that $q=2$, and $p=2^n$ then (\ref{eqn:Hau})
reproduces the optimal upper bound in Theorem~\ref
{thm:summarySILqPoissonNC}. In all other cases, however, both (\ref
{eqn:BrH}) and (\ref{eqn:Hau}) yield suboptimal bounds for $L^q$-spaces.

\begin{appendix}\label{app}
\section{Decoupling}
\label{app:decoupling}

In this appendix, we give a proof of the decoupling inequality (\ref
{eqn:decouplingSIIntro}). Recall that a Banach space $X$ is called a
\emph{UMD space} if for some (then, every) $1<p<\infty$ there is a
constant $C_{p,X}\ge0$ such that for any $X$-valued martingale
difference sequence $(d_n)_{n\geq1}$, any sequence of signs $(\varepsilon
_n)_{n\geq1}$ and any $N\geq1$ one has
%
\begin{equation}
\label{eqn:defUMD} \Biggl(\E \Biggl\| \sum_{n=1}^N
\varepsilon_n d_n \Biggr\|_X^p
\Biggr)^{{1}/{p}} \le C_{p,X} \Biggl(\E \Biggl\|\sum
_{n=1}^N d_n \Biggr\|_X^p
\Biggr)^{{1}/{p}}.
\end{equation}
It is well known that any $L^q$-space, classical or noncommutative, is
a UMD space if and only if $1<q<\infty$. We refer to \cite{Bur01} for
more information on UMD spaces.

The decoupling inequality (\ref{eqn:decouplingSIIntro}) is a direct
consequence of the following observation. For the convenience of the
reader, we reproduce its short proof, which appeared in \cite{Ver06},
Theorem 2.4.1 (see also \cite{Nee08}, Theorem 13.1).
%
\begin{lemma}
\label{lem:decouplingGeneral}
Let $1<p<\infty$ and let $X$ be a UMD Banach space. Consider a
filtration $(\cG_i)_{i=0}^n$ in $(\Om,\cF,\bP)$. Suppose that for every
$1\leq i\leq n$ we are given a $\cG_i$-measurable, mean-zero,
real-valued random variable $M_i$ which is independent of $\cG_{i-1}$
and, moreover, a $\cG_{i-1}$-measurable, $X$-valued random variable
$G_i$. Let $(M_{i}^c)_{i=1}^n$ be an independent copy of
$(M_i)_{i=1}^n$ on a probability space $(\Om_c,\cF_c,\bP_c)$. Then
%
\begin{equation}
\label{lem:decouplingGeneralLemma} \Biggl(\E\Biggl \|\sum_{i=1}^n
G_iM_i \Biggr\|_X^p
\Biggr)^{{1}/{p}} \leq C_{p,X} \Biggl(\E\E_c \Biggl\|\sum
_{i=1}^n G_iM_{i}^c
\Biggr\|_X^p \Biggr)^{{1}/{p}}.
\end{equation}
\end{lemma}
\begin{pf}
For $i=1,\ldots,n$ let $\cG_i^c$ be the sub-$\sigma$-algebra
generated by
$(M_j^c)_{j=1}^i$. Define
\[
d_{2i} = \tfrac{1}{2}G_i\bigl(M_i-M_{i}^c
\bigr),\qquad d_{2i-1}=\tfrac{1}{2}G_i\bigl(M_i+M_{i}^c
\bigr).
\]
We claim that $(d_i)_{i=1}^{2n}$ is a martingale difference sequence on
$\Om\ti\Om_c$ with respect to the filtration $(\cF_i)_{i=1}^{2n}$
defined by
\[
\cF_{2i} = \sigma\bigl(\cG_i,\cG_{i}^c
\bigr), \qquad\cF_{2i-1} = \sigma\bigl(\cG _{i-1},\cG
_{i-1}^c,M_{i}+M_{i}^c
\bigr)\qquad (i=1,\ldots,n).
\]
The result immediately follows from this claim and the UMD-property, since
\[
\sum_{i=1}^{2n} d_i = \sum
_{i=1}^n G_iM_i,\qquad
\sum_{i=1}^{2n} (-1)^{i+1}
d_i = \sum_{i=1}^n
G_iM_{i}^c.
\]
To prove the claim, note that $(d_i)_{i=1}^{2n}$ is adapted. Moreover,
by our assumptions on the $G_i$ and $M_i$,
\[
\E(d_{2i-1}|\cF_{2i-2}) = \tfrac{1}{2} G_i\E
\bigl(M_i+M_{i}^c|\cG _{i-1},\cG
_{i-1}^c\bigr) = 0
\]
and
\begin{eqnarray*}
\E(d_{2i}|\cF_{2i-1}) & = &\tfrac{1}{2} G_i
\E\bigl(M_i-M_{i}^c|\cG _{i-1},\cG
_{i-1}^c,M_{i}+M_{i}^c
\bigr)
\\
& =& \tfrac{1}{2} G_i\E\bigl(M_i-M_{i}^c|M_{i}+M_{i}^c
\bigr) = 0,
\end{eqnarray*}
where the final step follows from a direct computation, using that
$M_i$ and $M_{i}^c$ are independent and identically distributed.
\end{pf}
%
\begin{lemma}
\label{lem:decouplingSI}
Let $1<p<\infty$ and let $X$ be a UMD Banach space. Let $N$ be a
Poisson random measure on $\R_+\ti J$ and let $N^c$ be an independent
copy of $N$. Fix a filtration $(\cF_t)_{t>0}$ in $\Om$ satisfying
Assumption~\ref{ass:PoissonFiltration}. If $F$ is a simple, adapted
$X$-valued process, then for all $t>0$ and $B\in\cJ$,
%
\begin{equation}
\label{eqn:decouplingSI} \biggl(\E \biggl\|\int_{(0,t]\ti B} F \,d\tilde{N}
\biggr\|_{X}^p \biggr)^{1/p} \leq C_{p,X}
\biggl(\E\E_c \biggl\|\int_{(0,t]\ti B} F \,d\tilde
{N}^c\biggr \|_{X}^p \biggr)^{{1}/{p}},
\end{equation}
where $C_{p,X}$ is the constant in (\ref{eqn:defUMD}).
\end{lemma}
\begin{pf}
Let $F$ be the simple adapted process in (\ref{eqn:simpleBanach}). We
may assume that $t=t_{l+1}$ and $B=\bigcup_{j=1}^m A_j$. For every $1\leq
i\leq l$ and $1\leq j\leq m$, we set
\begin{eqnarray*}
G_{(i,j)} &= &\sum_{k=1}^n
F_{i,j,k}x_{i,j,k},\qquad M_{(i,j)} = \tilde
{N}\bigl((t_i,t_{i+1}]\ti A_j\bigr),\\
M_{(i,j)}^c& =& \tilde {N}^c\bigl((t_i,t_{i+1}]
\ti A_j\bigr).
\end{eqnarray*}
Under Assumption~\ref{ass:PoissonFiltration}, the subalgebras defined
for $i=1,\ldots,l$ and $j=1,\ldots,m$ by
\begin{eqnarray*}
\cG_{(i,j)} & =& \sigma \bigl(\cF_{t_i}, \tilde{N}
\bigl((t_i,t_{i+1}]\ti A_k\bigr), k=1,\ldots,j \bigr)\qquad
\mathrm{if}\ 1\leq j\leq m-1,
\\
\cG_{(i,m)} & =& \cF_{t_{i+1}}
\end{eqnarray*}
form a filtration if we equip the pairs $(i,j)$ with the lexicographic
ordering. Moreover, the sequences $(G_{(i,j)})_{(i,j)}$,
$(M_{(i,j)})_{(i,j)}$, and $(M_{(i,j)}^c)_{(i,j)}$ satisfy the
conditions of Lemma~\ref{lem:decouplingGeneral} and inequality (\ref
{eqn:decouplingSI}) exactly corresponds to the estimate (\ref
{lem:decouplingGeneralLemma}).
\end{pf}

\section{Proof of Theorem \texorpdfstring{\lowercase{\protect\ref{thm1.1}}}{1.1}: Remaining cases}
\label{app:remainingCases}

For completeness, we give a proof here of the remaining cases of
Theorem~\ref{thm1.1}. We continue to use the same
notation, in particular $\cI_{\mathrm{elem}}$ is the space of all
simple functions on $\Om\ti\R_+\ti J\ti S$ with support of finite
measure and $\cA$ denotes the sub-$\sigma$-algebra of $\cB(\R_+)\ti
\cJ$
generated by the sets $(t_i,t_{i+1}]\ti A_j$. Let us note that it
suffices to prove the upper estimates $\lesssim_{p,q}$ in (\ref
{eqn:summarySILqPoisson}). The reverse estimates then follow by the
duality argument presented in the case $1<p\leq q\leq2$.

\emph{Case $2\leq p\leq q\leq2$}: Fix $\eps>0$. By density of $\cI
_{\mathrm{elem}}$ in $\cD_{p,q}^p$ and $\cD_{q,q}^p$, we can find a
decomposition $F=F_1+F_2$ with $F_{\alpha} \in\cI_{\mathrm{elem}}$ for
$\al=1,2$ such that
\[
\|F_1\|_{\cD_{p,q}^p} + \|F_2\|_{\cD_{q,q}^p} \leq\|F
\|_{\cD
_{p,q}^p +
\cD_{q,q}^p} + \eps.
\]
We may assume that $F_{1}$ and $F_2$ have the same support in $\R_+\ti
J$ as $F$. By Lemma~\ref{lem:conditionSimple} $\E(F_{\alpha}|\cA)$ is
of the form
%
\begin{equation}
\label{eqn:FormFalphaCond} \E(F_{\alpha}|\cA) = \sum_{i,j,k}
F_{i,j,k,\alpha} \oto\chi _{(t_i,t_{i+1}]} \oto\chi_{A_j} \oto
x_{i,j,k,\alpha} \qquad(\alpha=1,2).
\end{equation}
Let $y_{i,j,\alpha} = \sum_{k=1}^n F_{i,j,k,\alpha} \oto
x_{i,j,k,\alpha
}$ and set $d_{i,j,\alpha} = y_{i,j,\alpha}N_{i,j}^c$, so that
\[
d_{i,j} = d_{i,j,1} + d_{i,j,2}.
\]
If we apply Theorem~\ref{thm:summaryRosIntro} pointwise in $\Om$ and
subsequently take $L^p(\Om)$-norms on both sides, we find
\begin{eqnarray*}
& &\biggl(\E\E_c \biggl\|\int_{(0,t]\ti B} F \,d
\tilde{N}^c \biggr\| _{L^q(S)}^p \biggr)^{{1}/{p}}
\\
&&\qquad = \biggl(\E\E_c \biggl\|\sum_{i,j}
d_{i,j} \biggr\|_{L^q(S)}^p \biggr)^{1/p}
\\
&&\qquad \lesssim_{p,q} \max \biggl\{ \biggl(\E \biggl\| \biggl(\sum
_{i,j} \E _{c}|d_{i,j}|^2
\biggr)^{{1}/{2}} \biggr\|_{L^q(S)}^p \biggr)^{{1}/{p}},
\\
&&\qquad\quad \hspace*{39pt}\biggl(\sum_{i,j} \E\E_c
\|d_{i,j,1}\|_{L^q(S)}^p \biggr)^{1/p} +
\biggl(\E \biggl(\sum_{i,j} \E_{c}
\|d_{i,j,2}\|_{L^q(S)}^q \biggr)^{p/q}
\biggr)^{{1}/{p}} \biggr\}
\\
&&\qquad \lesssim_{p,q} \max\bigl\{\|F\|_{\cS_q^p}, \|F_1
\|_{\cD_{p,q}^p} + \| F_2\| _{\cD_{q,q}^p}\bigr\} \leq\|F
\|_{\cI_{p,q}} + \eps,
\end{eqnarray*}
where the penultimate inequality follows by the computations in (\ref{eqn:normEstimatesFalpha}).

\emph{Case $1<p<2\leq q<\infty$}: Fix $\eps>0$. By density of $\cI
_{\mathrm{elem}}$ in $\cD_{p,q}^p$ and $\cS_q^p\cap\cD_{q,q}^p$, we
can find a decomposition $F=F_1+F_2$ with $F_{\alpha} \in\cI
_{\mathrm
{elem}}$ for $\al=1,2$ such that
\[
\|F_1\|_{\cD_{p,q}^p} + \|F_2\|_{\cS_q^p\cap\cD_{q,q}^p} \leq
\|F\| _{\cI_{p,q}} + \eps.
\]
We may assume that $F_{1}$ and $F_2$ have the same support in $\R_+\ti
J$ as $F$. By Lemma~\ref{lem:conditionSimple}, $\E(F_{\alpha}|\cA)$ is
of the form (\ref{eqn:FormFalphaCond}). Let $y_{i,j,\alpha} = \sum_{k=1}^n F_{i,j,k,\alpha} \oto x_{i,j,k,\alpha}$ and set
$d_{i,j,\alpha
} = y_{i,j,\alpha}N_{i,j}^c$, so that
\[
d_{i,j} = d_{i,j,1} + d_{i,j,2}.
\]
We apply Theorem~\ref{thm:summaryRosIntro} pointwise in $\Om$ and
subsequently take $L^p(\Om)$-norms on both sides to find
\begin{eqnarray*}
&& \biggl(\E\E_c \biggl\|\int_{(0,t]\ti B} F \,d
\tilde{N}^c \biggr\| _{L^q(S)}^p \biggr)^{{1}/{p}}
\\
&&\qquad = \biggl(\E\E_c \biggl\|\sum_{i,j}
d_{i,j} \biggr\|_{L^q(S)}^p \biggr)^{1/p}
\\
&&\qquad \lesssim_{p,q} \biggl(\sum_{i,j} \E
\E_c\|d_{i,j,1}\| _{L^q(S)}^p
\biggr)^{{1}/{p}}
\\
&&\hspace*{12pt}\qquad\quad{} + \max \biggl\{ \biggl(\E \biggl\| \biggl(\sum_{i,j} \E
_{c}|d_{i,j,2}|^2 \biggr)^{{1}/{2}}
\biggr\|_{L^q(S)}^p \biggr)^{{1}/{p}},
\\
&&\hspace*{63pt}\qquad\quad{} \biggl(\E \biggl(\sum_{i,j} \E_{c}
\|d_{i,j,2}\|_{L^q(S)}^q \biggr)^{p/q}
\biggr)^{{1}/{p}} \biggr\}
\\
&&\qquad \lesssim_{p,q} \|F_1\|_{\cD_{p,q}^p} + \max\bigl\{
\|F_2\|_{\cS_q^p},\| F_2\| _{\cD_{q,q}^p}\bigr\} \leq\|F
\|_{\cI_{p,q}} + \eps,
\end{eqnarray*}
where the penultimate inequality follows by (\ref{eqn:normEstimatesFalpha}).

\emph{Case $1<q<2\leq p<\infty$}: Let $\eps>0$. By density of $\cI
_{\mathrm{elem}}$ in $\cS_q^p$ and $\cD_{q,q}^p$, we can find a
decomposition $F=F_1+F_2$ with $F_{\alpha} \in\cI_{\mathrm{elem}}$ for
$\al=1,2$ such that
\[
\|F_1\|_{\cS_q^p} + \|F_2\|_{\cD_{q,q}^p} \leq\|F
\|_{\cS_q^p + \cD
_{q,q}^p} + \eps.
\]
We may assume that $F_{1}$ and $F_2$ have the same support in $\R_+\ti
J$ as $F$. By Lemma~\ref{lem:conditionSimple} $\E(F_{\alpha}|\cA)$ is
of the form (\ref{eqn:FormFalphaCond}). Let $y_{i,j,\alpha} = \sum_{k=1}^n F_{i,j,k,\alpha} \oto x_{i,j,k,\alpha}$ and set
$d_{i,j,\alpha
} = y_{i,j,\alpha}N_{i,j}^c$, so that
\[
d_{i,j} = d_{i,j,1} + d_{i,j,2}.
\]
We apply Theorem~\ref{thm:summaryRosIntro} pointwise in $\Om$ and
subsequently take $L^p(\Om)$-norms on both sides to obtain
\begin{eqnarray*}
&& \biggl(\E\E_c \biggl\|\int_{(0,t]\ti B} F \,d
\tilde{N}^c \biggr\| _{L^q(S)}^p \biggr)^{{1}/{p}}
\\
&&\qquad = \biggl(\E\E_c \biggl\|\sum_{i,j}
d_{i,j} \biggr\|_{L^q(S)}^p \biggr)^{1/p}
\\
&&\qquad \lesssim_{p,q} \max \biggl\{ \biggl(\sum
_{i,j} \E\E_c\|d_{i,j}\|
_{L^q(S)}^p \biggr)^{{1}/{p}},
\\
&&\hspace*{39pt}\qquad\quad \biggl(\E \biggl\| \biggl(\sum_{i,j}
\E_{c}|d_{i,j,1}|^2 \biggr)^{1/2}
\biggr\|_{L^q(S)}^p \biggr)^{{1}/{p}} \\
&&\hspace*{50pt}\qquad{}+ \biggl(\E \biggl(\sum
_{i,j} \E _{c}\|d_{i,j,2}
\|_{L^q(S)}^q \biggr)^{{p}/{q}} \biggr)^{1/p}
\biggr\}
\\
&&\qquad \lesssim_{p,q} \max\bigl\{\|F\|_{\cD_{p,q}^p}, \|F_1
\|_{\cS_q^p} + \| F_2\| _{\cD_{q,q}^p}\bigr\} \leq\|F
\|_{\cI_{p,q}} + \eps,
\end{eqnarray*}
where the penultimate inequality follows by the computations in (\ref
{eqn:normEstimatesFalpha}).

\emph{Case $1<q\leq p\leq2$}: Fix $\eps>0$. By density of
$\cI
_{\mathrm{elem}}$ in $\cS_q^p$ and $\cD_{q,q}^p\cap\cD_{p,q}^p$, we
can find a decomposition $F=F_1+F_2$ with $F_{\alpha} \in\cI
_{\mathrm
{elem}}$ for $\al=1,2$ such that
\[
\|F_1\|_{\cS_q^p} + \|F_2\|_{\cD_{q,q}^p\cap\cD_{p,q}^p} \leq
\|F\| _{\cI_{p,q}} + \eps.
\]
We may assume that $F_{1}$ and $F_2$ have the same support in $\R_+\ti
J$ as $F$. By Lemma~\ref{lem:conditionSimple}, $\E(F_{\alpha}|\cA)$ is
of the form (\ref{eqn:FormFalphaCond}). Let $y_{i,j,\alpha} = \sum_{k=1}^n F_{i,j,k,\alpha} \oto x_{i,j,k,\alpha}$ and set
$d_{i,j,\alpha
} = y_{i,j,\alpha}N_{i,j}^c$, so that
\[
d_{i,j} = d_{i,j,1} + d_{i,j,2}.
\]
We apply Theorem~\ref{thm:summaryRosIntro} pointwise in $\Om$ and
subsequently take $L^p(\Om)$-norms on both sides to find
\begin{eqnarray*}
&& \biggl(\E\E_c \biggl\|\int_{(0,t]\ti B} F \,d
\tilde{N}^c \biggr\| _{L^q(S)}^p \biggr)^{{1}/{p}}
\\
&&\qquad = \biggl(\E\E_c \biggl\|\sum_{i,j}
\,d_{i,j} \biggr\|_{L^q(S)}^p \biggr)^{1/p}
\\
&&\qquad \lesssim_{p,q} \biggl(\E \biggl\| \biggl(\sum_{i,j}
\E _{c}|d_{i,j,1}|^2 \biggr)^{{1}/{2}}
\biggr\|_{L^q(S)}^p \biggr)^{{1}/{p}}
\\
&&\qquad\qquad{} + \max \biggl\{ \biggl(\E \biggl(\sum_{i,j}
\E_{c}\|d_{i,j,2}\| _{L^q(S)}^q
\biggr)^{{p}/{q}} \biggr)^{{1}/{p}},
\\
&&\hspace*{51pt}\qquad\quad \biggl(\E \biggl\| \biggl(\sum_{i,j}
\E_{c}|d_{i,j,2}|^2 \biggr)^{1/2}
\biggr\|_{L^q(S)}^p \biggr)^{{1}/{p}} \biggr\}
\\
&&\qquad \lesssim_{p,q} \|F_1\|_{\cS_q^p} + \max\bigl\{
\|F_2\|_{\cD_{q,q}^p}, \| F_2\|_{\cD_{p,q}^p}\bigr\} \leq\|F
\|_{\cI_{p,q}} + \eps,
\end{eqnarray*}
where the penultimate inequality follows as in (\ref
{eqn:normEstimatesFalpha}). This completes the proof.
\end{appendix}

\section*{Acknowledgements}
It is a pleasure to thank Sonja Cox, Jan Maas, Jan van
Neerven and Mark Veraar for helpful discussions on the topic of this
paper. I would like to thank the anonymous reviewers for their detailed
comments, which substantially improved the presentation of the results.

%


\printaddresses

\end{document}